\pdfoutput=1
\documentclass[a4paper,11pt]{article}
\usepackage[english]{babel}
\usepackage{amsfonts}
\usepackage{amsthm}
\usepackage{amsmath}
\usepackage{amssymb}
\usepackage{enumerate}
\usepackage{mathtools}
\usepackage{tikz-cd}
\usepackage[all]{xy}
\usepackage{graphicx}
\usepackage{amsmath}
\usepackage{adjustbox}
\usepackage{graphicx}
\usepackage{extpfeil}
\usepackage{comment}
\usepackage{float}
\usepackage[labelformat=empty]{caption}
\usepackage[toc]{appendix}
\usepackage[left=3.4cm,top=3cm,right=3.4cm,bottom=3cm]{geometry}
\usepackage{hyperref}
\usepackage{resizegather}
\usepackage[activate={true, nocompatibility}, final, tracking=true, kerning=true, spacing=true, factor=1100, stretch=10, shrink=10]{microtype}
\usepackage{authblk}
\usepackage{cleveref}

\theoremstyle{plain}
\newtheorem{thm}{Theorem}[section]

\newtheorem{defn}[thm]{Definition}
\newtheorem{lemma}[thm]{Lemma}

\newtheorem{example}[thm]{Example}
\newtheorem{prop}[thm]{Proposition}

\newtheorem{remark}[thm]{Remark}

\newtheorem*{thm*}{Main Theorem}

\tikzset{
  symbol/.style={
    draw=none,
    every to/.append style={
      edge node={node [sloped, allow upside down, auto=false]{$#1$}}}
  }
}

\microtypecontext{spacing=nonfrench}

\begin{document}
\thispagestyle{empty}

\title{\textbf{On the quasicompactness of the moduli stack of logarithmic $G$-connections over a curve}}

\author{Andres Fernandez Herrero}
\date{}
\maketitle
\begin{abstract}
Fix a smooth projective curve over a field of characteristic zero and a finite set of punctures. Let $G$ be a connected linear algebraic group. We prove that the moduli of $G$-bundles with logarithmic connections having fixed residue classes at the punctures is an algebraic stack of finite type.
\end{abstract}

\tableofcontents
\begin{section}{Introduction}
Fix a curve $C$ that is smooth projective and geometrically irreducible over a field $k$ of characteristic $0$. Let $G$ be a connected linear algebraic group over $k$. If $k = \mathbb{C}$, then the curve $C$ is a Riemann surface. The Riemann-Hilbert correspondence establishes an equivalence between the groupoid of regular $G$-connections on $C$ and the groupoid of homomorphism from the topological fundamental group $\pi_1(C)$ into $G(\mathbb{C})$. 

One would like to view the moduli stack $\text{Conn}_{G}(C)$ of regular $G$-connections on $C$ as an algebraic avatar of the moduli of representations of $\pi_1(C)$. Motivated by this, we show that $\text{Conn}_{G}(C)$ is an algebraic stack of finite type for an arbitrary connected linear algebraic group $G$ over any field of characteristic $0$.

More generally, the Riemann-Hilbert correspondence can be established for a noncompact Riemann surface. Fix a finite set $\{x_i\}$ of $k$-points $x_i \in C$. Set $D= \sum_i x_i$ to be the corresponding reduced divisor. In this case one would like to study representations of the fundamental group of $C \setminus D$. The corresponding objects considered in this context are logarithmic connections with poles at $D$ \cite{deligne.regulier}. This logarithmic version of the Riemann-Hilbert correspondence gives an analytic description of the monodromy around a puncture $x_i$, meaning the conjugacy class of the element of $G(\mathbb{C})$ determined by a loop around $x_i$. Namely, the monodromy around $x_i$ is related to the residue $O_i$ of the corresponding logarithmic connection at $x_i$. Here $O_i$ is an adjoint orbit in the Lie algebra of $G$.

The notion of logarithmic connection is algebraic; there is an algebraic stack $\text{Conn}_{G}^{D}(C)$ parametrizing logarithmic $G$-connections on $C$ with poles at $D$. In view of the Riemann-Hilbert correspondence, it is natural to view this stack as an algebraic avatar of representations of the fundamental group of $C \setminus D$. The main goal of this paper is to study some geometric properties of $\text{Conn}_{G}^{D}(C)$. To this end, we prove the following theorem.
\begin{thm*}[=Theorem \ref{thm: quasicompactness logarithmic connections}]
Let $k$ be a field of characteristic $0$. Let $G$ be a connected linear algebraic group over $k$. Let $\text{Conn}_{G}^{D, O_i}(C)$ denote the moduli stack of $G$-logarithmic connections with poles at $D$ and fixed residue class $O_i$ at each $x_i$. Then $\text{Conn}_{G}^{D, O_i}(C)$ is of finite type over $k$.
\end{thm*}
The boundedness of a related moduli problem for $G = \text{GL}_n$ was proven by Inaba, Iwasaki and Saito \cite{inaba-parabolic}. In that paper they assume that the logarithmic connection comes equipped with a compatible parabolic vector bundle that is semistable with respect to a fixed set of weights.

If $D$ is not empty, we prove in Proposition \ref{prop: characterization of quasicompactness} that $\text{Conn}_G^D(C)$ is quasicompact if and only if $G$ is unipotent. This shows that in most cases of interest it is necessary to restrict to a substack of $\text{Conn}_{G}^{D}(C)$ in order to expect a bounded moduli problem. Theorem \ref{thm: quasicompactness logarithmic connections} says that the substack $\text{Conn}^{D, O_i}_{G}(C)$ with fixed residues is quasicompact. In a different direction, Nitsure \cite{nitsure-logarithmic}[Prop. 3.1] proved that the substack of semistable connections inside $\text{Conn}_G^D(C)$ is quasicompact when $G = \text{GL}_n$. He uses this to construct a quasiprojective coarse moduli space for semistable logarithmic connections \cite{nitsure-logarithmic}[Thm. 3.5].

The first part of this paper deals with the case when the set of punctures $\{x_i\}$ is empty. Then the relevant geometric object is the stack $\text{Conn}_G(C)$ of regular $G$-connection. We prove Proposition \ref{prop: quasicompactness regular connections}, which is a special case of Theorem \ref{thm: quasicompactness logarithmic connections} above. The Harder-Narasimhan filtration for vector bundles is our main tool to show quasicompactness. The proof Proposition \ref{prop: quasicompactness regular connections} is different than for the more general Theorem \ref{thm: quasicompactness logarithmic connections}; it yields better bounds for the Harder-Narasimhan type of the underlying bundle. In the case when $G = \text{GL}_n$, Simpson also proved that $\text{Conn}_{G}(C)$ is of finite type \cite{Simpson-repnI}[Cor. 3.4]. This was one of the main ingredients for his construction of the deRham moduli space in \cite{Simpson-repnI} \cite{simpson-repnII}. See pages 15-16 for a comparison of our approach with Simpson's in the case of $\text{GL}_n$.

The second part of the paper deals with the case when the set of punctures $\{x_i\}$ is nonempty. This part is where we prove our main theorem (Theorem \ref{thm: quasicompactness logarithmic connections}). We also prove an auxiliary result that could be of independent interest. This is Proposition \ref{prop: atiyah class with prescribed residues}. It gives an explicit description of the cohomological obstruction for a $G$-bundle to admit a logarithmic connection with some given set of residues. The proof of Proposition \ref{prop: atiyah class with prescribed residues} is completely algebraic and applies to families over an arbitrary base without any assumptions on the characteristic of the ground field. It generalizes \cite{biswas-criterion-log-prescribed}[Prop. 3.1], which was proven for the group $\text{GL}_n$ over $\mathbb{C}$ using transcendental methods.
\end{section}
\begin{section}{Preliminaries}
\begin{subsection}{Notation}
We work over a fixed perfect ground field $k$. Some results will hold only when the characteristic of $k$ is $0$; we will explicitly mention when this is the case. Unless otherwise stated, all schemes will be understood to be schemes over $k$. An undecorated product of $k$-schemes (e.g. $X\times S$) should always be interpreted as a fiber product over $k$. We will sometimes write $X_S$ instead of $X \times S$. If $R$ is a $k$-algebra and $S$ is a $k$-scheme, we may use the notation $S_R$ to denote the fiber product $S \times \text{Spec}(R)$. If a scheme $t$ is Spec of a field, we will write $\kappa(t)$ for the corresponding coordinate field.

We fix once and for all a curve $C$ that is smooth, projective and geometrically connected over $k$. We let $g$ denote the genus of $C$. We also fix an ample line bundle $\mathcal{O}_{C}(1)$ on $C$. Choose a a finite set $\{x_i\}_{i \in I}$ of $k$-points in $C$. We will denote by $q_i: x_i \rightarrow C$ the closed immersion of $x_i$ into $C$.

Let $G$ be a smooth connected linear algebraic group over $k$. Write $\mathfrak{g}= \text{Lie}(G)$ for the Lie algebra of $G$. There is a representation $Ad: G  \longrightarrow \text{GL}(\mathfrak{g})$ called the adjoint representation. For $G = \text{GL}_n$ it is given by matrix conjugation.

G-bundles will play a prominent role in this paper. Let $X$ be a $k$-scheme. A $G$-bundle over $X$ is a scheme $\pi: \mathcal{P} \rightarrow X$ equipped with a right $G$-action. This action is required to make $\mathcal{P}$ a $G$-torsor in the \'etale topology. This means that for some \'etale cover $T \rightarrow X$ the base-change $\mathcal{P}\times_X T$ is $G$-equivariantly isomorphic to the trivial $G$-torsor $G_T$. An isomorphism between $G$-bundles is an isomorphism of $X$-schemes that intertwines the $G$-actions.

We will often deal with quasicoherent sheaves on schemes. Let $X$ and $Y$ be schemes and let $\mathcal{Q}$ be quasicoherent sheaf on $Y$. If there is a clear implicit choice of morphism $f: X \rightarrow Y$, we will write $\mathcal{Q}|_{X}$ to denote the pullback $f^{*}\mathcal{Q}$ of the quasicoherent sheaf $\mathcal{Q}$ by $f$. The same convention will be used for pullbacks of $G$-bundles.
\end{subsection}
\begin{subsection}{A representability lemma}
Let $X, S$ be $k$-schemes. Suppose that we have a morphism $f: X \longrightarrow S$.
\begin{defn}
 Let $\varphi: A \to X$ be a scheme over $X$. Define $\Gamma_{X /S}(A)$ to be the functor from $S$-schemes to sets given as follows. For any $S$-scheme $g: T \longrightarrow S$, we set
 \begin{gather*}
     \Gamma_{X/S}(A)\, (T) \vcentcolon = \left\{ \; \text{sections $s: X \times_S T \to A \times_S T$ of the morphism $\varphi_{S}: A \times_S T \to X \times_S T$} \, \right\} 
 \end{gather*}
\end{defn}
\begin{lemma} \label{lemma: representability of sections}
Let $X \longrightarrow S$ be a proper flat morphism of finite presentation.  Let $A \to X$ be an affine morphism of finite presentation. The functor $\Gamma_{X/S}(A)$ is represented by a scheme that is relatively affine and of finite presentation over $S$.
\end{lemma}
\begin{proof}
This follows from \cite[Thm. 2.3 (i)]{hall-rydh-hilbert-quot} and \cite[Thm. 1.3 (ii)]{hall-rydh-tannaka}. 
\end{proof}

\end{subsection}
\begin{subsection}{Change of group for $G$-bundles}
Let $\text{B}G$ denote the classifying stack of $G$. This is the pseudofunctor from $k$-schemes to groupoids given as follows. For any $k$-scheme $S$,
\[ \text{B}G \, (S) \; \vcentcolon = \; \left\{ \begin{matrix} \text{groupoid of  $G$-bundles $\mathcal{P}$}\;  \text{over $S$} \end{matrix} \right\}  \]
We have that $\text{B}G$ is an algebraic stack \cite[\href{https://stacks.math.columbia.edu/tag/06PL}{Tag 06PL}]{stacks-project}, and so in particular it is a stack for the fppf topology. Let $\text{Bun}_{G}(C)$ denote the moduli stack of $G$-bundles on $C$. This is defined by $\text{Bun}_{G}(C) \, (S) \vcentcolon = \text{B}G (C \times S)$ for any $k$-scheme $S$. Since $\text{B}G$ is a stack in the fppf topology, it directly follows from the definition that $\text{Bun}_{G}(C)$ also satisfies fppf descent.

\begin{prop}(\cite{behrend-thesis}[Prop. 4.4.4]). \label{prop: quasicompactness associated bundle map}
Suppose that $G$ is a connected reductive group over $k$. 
\begin{enumerate}[(i)]
    \item Let $\rho: G \longrightarrow \text{GL}_n$ be a faithful representation. Let $\rho_{*}: \text{Bun}_G(C) \longrightarrow \text{Bun}_{\text{GL}_n}(C)$ be the morphism induced by extension of structure groups. Then, $\rho_{*}$ is schematic, affine and of finite type.
    \item $\text{Bun}_{G}(C)$ is an algebraic stack locally of finite type over $k$.
\end{enumerate}
\end{prop}
\begin{proof}
\quad \newline
\noindent (i) Let $S$ be a $k$-scheme. Let $\mathcal{P}: S \longrightarrow \text{Bun}_{\text{GL}_n}(C)$ be a $\text{GL}_n$-bundle over $C \times S$. We want to show that $\text{Bun}_G(C) \times_{\text{Bun}_{\text{GL}_n}(C)} S\to S$ is represented by a scheme that is relatively affine and of finite type over $S$.
	
Matsushima's criterion \cite{richardson-affinehomegeneous} implies that the homogeneous space $\text{GL}_n \, / \, G$ is affine. Form the associated fiber bundle $X \vcentcolon = \mathcal{P} \times^{\text{GL}_n} \left(\text{GL}_{n \, C \times S} \, / \, G_{C \times S}\right)$, which is affine and of finite presentation over $S$ by \'etale descent. For any $S$-scheme $f: T \longrightarrow S$, we have
\begin{align*}
    (\text{Bun}_G(C) \times_{\text{Bun}_{\text{GL}_n}(C)} S)(T) & = \left\{ \; \text{sections of the morphism} \; X \times_S T \longrightarrow C \times T \; \right\}\\
    & = \Gamma_{C\times S/S}(X)
\end{align*}
By \Cref{lemma: representability of sections}, this is represented by a relatively affine scheme of finite type over $S$.
\medskip

\noindent (ii) By part $(i)$, it suffices to show that $\text{Bun}_{\text{GL}_{n}}(C)$ is an algebraic stack that is locally of finite type over $k$. The stack $\text{Bun}_{\text{GL}_n}(C)$ classifies rank $n$ vector bundles on $C$. This is an algebraic stack locally of finite type over $k$, see \cite{neumann-vectorbundles}[Thm. 2.57], \cite{heinloth-vectorbundles}[Example 1.14] or \cite{lmb-champsalgebriques}[4.6.2.1].
\end{proof}

We now return to the general case when $G$ is a smooth connected linear algebraic group. Let $U$ denote the unipotent radical of $G$. Recall that this is the biggest smooth connected normal unipotent subgroup of $G$. The quotient $G/U$ is a reductive linear algebraic group. Let $\rho_{G}: G \rightarrow G/U$ denote the quotient morphism.
\begin{prop} \label{prop: quasicompactness associated bundle morphism unipotent}
Suppose that the characteristic of $k$ is $0$. Let $(\rho_{G})_*: \text{Bun}_{G}(C) \rightarrow \text{Bun}_{G/U}(C)$ denote the morphism induced by extension of structure groups.
\begin{enumerate}[(i)]
    \item For all $k$-schemes $T$ and morphisms $T \rightarrow \text{Bun}_{G/U}(C)$, the fiber product $\text{Bun}_{G}(C) \times_{\text{Bun}_{G/U}(C)} T$ is an algebraic stack of finite type over $T$.
    
    \item $\text{Bun}_{G}(C)$ is an algebraic stack locally of finite type over $k$.
    
    \item The morphism of algebraic stacks $(\rho_{G})_*$ is of finite type.
    
    \item If the group $G$ is unipotent, then $\text{Bun}_{G}(C)$ is an algebraic stack of finite type over $k$.
\end{enumerate}
\end{prop}
\begin{proof}
\quad \newline
\noindent (i) We will induct on the length $l(U)$ of the nilpotent series for the unipotent group $U$. The base case is $l(U)=0$. Then $U$ is trivial and so $(\rho_{G})_*$ is the identity. Hence $(i)$ holds trivially for the base case.
    
    Suppose that (i) holds whenever $l(U) \leq n$. Assume that $l(U) = n+1$. Let $Z_{U}$ denote the neutral component of the center of $U$. The algebraic group $Z_{U}$ is a smooth connected normal subgroup of $G$. Let $\rho_{z}$ denote the quotient morphism $\rho_{z}: G \rightarrow G/Z_{U}$. The unipotent radical of $G/Z_{G}(U)$ is $U/Z_{G}(U)$. The morphism $(\rho_{G})_{*}$ factors as
    \[ (\rho_{G})_{*} : \text{Bun}_{G}(C) \xrightarrow{(\rho_{z})_{*}} \text{Bun}_{G/Z_{U}}(C) \xrightarrow{(\rho_{G/Z_{U}})_{*}} \text{Bun}_{G/U}(C)\]
    Note that $l(U/Z_{U}) = n$ by construction. By the induction hypothesis, (i) holds for the morphism $(\rho_{G/Z_{U}})_*$. Therefore it suffices to show that the morphism $(\rho_{z})_*$ satisfies (i).
    
    Let $T$ be a $k$-scheme. Choose a morphism $T \rightarrow \text{Bun}_{G/Z_{U}}(C)$ represented by a $G/Z_{U}$-bundle $\mathcal{P}$ on $C_{T}$. We want to show that the fiber product $\text{Bun}_{G}(C) \times_{\text{Bun}_{G/Z_{U}}(C)} T$ is an algebraic stack of finite type over $T$. After passing to a Zariski cover of $T$, we can assume that $T$ is affine. 
    
    Let $(Z_{U})^{\mathcal{P}}$ denote the unipotent commutative group scheme over $C_{T}$ obtained by twisting $Z_{U}$ by $\mathcal{P}$ via the natural conjugation action of $G/Z_{U}$ on $Z_{U}$. Since the characteristic of $k$ is $0$, the commutative unipotent group $Z_{U}$ is isomorphic to a product of additive groups $(\mathbb{G}_a)^r$. Furthermore $G/Z_{U}$ acts by linear automorphisms on $Z_{U} \cong (\mathbb{G}_a)^r$. By \'etale descent for vector bundles, it follows that the twisted group $(Z_{U})^{\mathcal{P}}$ is the total space of a vector bundle $\mathcal{V}$ on $C_{T}$.
    
    By \cite{hoffmann_Gbundles}[Prop. 3.1], there is an obstruction in $H^2_{fppf}(C_{T}, \mathcal{V})$ for lifting the $G/Z_{U}$-bundle $\mathcal{P}$ to a $G$-bundle. Since $\mathcal{V}$ is a vector bundle, the fppf cohomology group $H^2_{fppf}(C_{T}, \mathcal{V})$ is the same as the usual sheaf cohomology group $H^2(C_{T}, \mathcal{V})$ in the small Zariski site (cf. \cite{hoffmann_Gbundles}[Remark 3.3]). Since the projective morphism $C_{T} \rightarrow T$ has relative dimension $1$ and $T$ is affine, the Leray spectral sequence implies that $H^2(C_{T}, \mathcal{V}) = 0$. Now \cite{hoffmann_Gbundles}[Prop. 3.1] shows that the fiber product $\text{Bun}_{G}(C) \times_{\text{Bun}_{G/Z_{U}}(C)} T$ is isomorphic to the stack $\text{Bun}_{(Z_{U})^{\mathcal{P}}}(C_{T})$ classifying $(Z_{U})^{\mathcal{P}}$-bundles on $C_{T}$. We are reduced to showing that $\text{Bun}_{(Z_{U})^{\mathcal{P}}}(C_{T})$ is an algebraic stack of finite type over $T$. 
    
    Let $F^{\bullet}$ be the Grothendieck complex of $\mathcal{V}$ with respect to the morphism $C_{T} \longrightarrow T$ \cite[\href{https://stacks.math.columbia.edu/tag/0B91}{Tag 0B91}]{stacks-project}. Since the relative dimension of the projective morphism $C_{T} \rightarrow T$ is $1$, we can choose $F^{\bullet}$ to be a two term complex $\left[ F^0 \rightarrow F^1 \right]$, where $F^0$ and $F^1$ are finite free modules on the affine scheme $T$. We will also denote by $F^0$ and $F^1$ the corresponding (constant) vector group schemes over $T$. Note that $F^0$ acts on $F^1$ via the differential morphism $F^0 \rightarrow F^1$. A similar reasoning as in \cite{hoffmann_Gbundles}[Remark 2.2] shows that $\text{Bun}_{(Z_{U})^{\mathcal{P}}}(C_{T})$ is isomorphic to the quotient stack $\left[ F^1 / \, F^0\right]$. This is an algebraic stack of finite type over $T$, as desired.
    \medskip
    
    \noindent (ii) In order to show that $\text{Bun}_{G}(C)$ is an algebraic stack locally of finite presentation over $k$, we are allowed to check after passing to a smooth cover of the fppf stack $\text{Bun}_{G}(C)$. Indeed, the fact that algebraicity can be checked smooth locally follows from \cite[\href{https://stacks.math.columbia.edu/tag/06DC}{Tag 06DC}]{stacks-project} applied to an atlas of the source of the covering, and being locally of finite presentation can be checked smooth locally by \cite[\href{https://stacks.math.columbia.edu/tag/06Q9}{Tag 06Q9}]{stacks-project}. By Proposition \ref{prop: quasicompactness associated bundle map}[(ii)], $\text{Bun}_{G/U}(C)$ is an algebraic stack locally of finite type over $k$. This implies that $\text{Bun}_{G/U}(C)$ admits an atlas $A \rightarrow \text{Bun}_{G/U}(C)$ such that $A$ is locally of finite type over $k$. Consider the fiber product:
    \begin{figure}[H]
\centering
\begin{tikzcd}
     A \times_{\text{Bun}_{G/U}(C)} \text{Bun}_{G}(C) \ar[d] \ar[r]  &  \text{Bun}_{G}(C) \ar[d] \\
     A \ar[r] & \text{Bun}_{G/U}(C)
\end{tikzcd}
\end{figure}
\vspace{-0.5cm}
The morphism $A \times_{\text{Bun}_{G/U}(C)} \text{Bun}_{G}(C) \rightarrow \text{Bun}_{G}(C)$ is a smooth cover of $\text{Bun}_{G}(C)$. By part (i), $A \times_{\text{Bun}_{G/U}(C)} \text{Bun}_{G}(C)$ is an algebraic stack of finite type over $A$. Since $A$ is locally of finite type over $k$, this implies that $A \times_{\text{Bun}_{G/U}(C)} \text{Bun}_{G}(C) \rightarrow \text{Bun}_{G}(C)$ is locally of finite type over $k$, as desired.
\medskip

\noindent (iii) This is just a restatement of (i).
\medskip

\noindent (iv) If $G$ is unipotent, then $G/U$ is the trivial group. Therefore $\text{Bun}_{G/U}(C) = Spec(k)$, and $(\rho_{G})_*$ is the structure morphism of $\text{Bun}_{G}(C)$. By part (iii), it follows that $\text{Bun}_{G}(C)$ is of finite type over $k$.
\end{proof}
\end{subsection}
\end{section}
\begin{section}{The moduli stack of flat $G$-connections}
\begin{subsection}{Flat $G$-connections}
	There is a canonical left-invariant $\mathfrak{g}$-valued $1$-form on $G$ called the Maurer-Cartan form. It is defined as follows. Left translation induces an isomorphism $T_{G} \cong \mathfrak{g} \otimes \mathcal{O}_{G}$ for the tangent sheaf $T_G$ of $G$. This yields a chain of isomorphisms  $\Omega^{1}_{G / k} \otimes \mathfrak{g} = \text{Hom}_{\mathcal{O}_{G}} ( T_{G}, \mathcal{O}_{G}) \otimes \mathfrak{g} \cong \text{Hom}_{k}(\mathfrak{g}, \mathfrak{g}) \otimes \mathcal{O}_{G}$. The invariant $\mathfrak{g}$-valued 1-form on $G$ that corresponds to $\text{id}_{\mathfrak{g}} \otimes 1$ under this isomorphism is the Maurer-Cartan form. We denote it by $\omega \in   H^0\left(G, \, \Omega_{G / k}^{1} \otimes \mathfrak{g}\right)$. 
	
	Let $S$ be a $k$-scheme. The same construction can be applied to define the relative Maurer-Cartan form $\omega_S \in H^0\left(G_S, \,  \Omega^1_{G\times S/ S}\otimes \mathfrak{g} \right) $ for the group scheme $G_S$ over the base $S$. There exists a group scheme 
	\[\text{Aff}\left(\Omega^1_{C\times S/ S}\otimes \mathfrak{g} \right) = \text{GL}\left(\Omega^1_{C\times S/ S}\otimes \mathfrak{g} \right) \ltimes \left(\Omega^1_{C\times S/ S}\otimes \mathfrak{g}\right)\]
	over $C \times S$ that classifies affine transformations of the locally free sheaf $\Omega^1_{C\times S/ S}\otimes \mathfrak{g}$. We consider the restriction of this group scheme to the small \'etale site of the scheme $C \times S$, which is a sheaf that we shall also denote by $\text{Aff}\left(\Omega^1_{C\times S/ S}\otimes \mathfrak{g} \right)$. The Maurer-Cartan form can be used to define a homomorphism $\phi: G_{C \times S} \longrightarrow \text{Aff}\left(\Omega^1_{C\times S/ S}\otimes \mathfrak{g} \right)$ of sheaves of groups in the small \'etale site of $C \times S$ as follows. Given an \'etale $C\times S$-scheme $U$ and an element $g \in G_S(U)$, we can pull back the relative Maurer-Cartan form in order to obtain $g^{*}\omega_S \in  H^0\left(U,  \, \Omega_{U / S}^{1}  \otimes  \mathfrak{g}  \right)$. We define $\phi(U) \, (g) \vcentcolon = Id_{\Omega^1_{U/S}} \underset{\mathcal{O}_{U}}{\otimes} Ad(g) \, + \, (g^{-1})^{*} \omega_S$. Notice that it is important here that $U \to C \times S$ is \'etale, in order to identify $\Omega^1_{U/S}$ with the pullback of $\Omega^1_{C \times S/ S}$. The fact that this is a homomorphism follows from the left-invariance of the Maurer-Cartan form.
	
	Let $\mathcal{P}$ be a $G$-bundle on $C \times S$. Suppose that $X$ is a relatively affine $C \times S$-scheme equipped with an action of $G_{C \times S}$. The group scheme $G_{C \times S}$ acts diagonally on the product $\mathcal{P} \times_{C \times S} X$. We define the associated bundle $\mathcal{P} \times^{G} X$ to be the quotient $(\mathcal{P} \times_{C \times S} X )\, / \, G_{C \times S}$. \'Etale descent for affine morphisms \cite[\href{https://stacks.math.columbia.edu/tag/0245}{Tag 0245}]{stacks-project} implies that $\mathcal{P} \times^{G} X$ is represented by a scheme. One such example is the vector bundle $\mathcal{P} \times^{G} \mathfrak{g}_{C \times S}$ associated to the adjoint representation $Ad: G_{C \times S} \longrightarrow \text{GL}(\mathfrak{g}_{C \times S})$. This is called the adjoint bundle of $\mathcal{P}$, and will henceforth be denoted by $Ad\, \mathcal{P}$.
	
	We now recall the definition of $G$-connection on $\mathcal{P}$ relative to $S$. The homomorphism $\phi$ induces an action of $G_{C \times S}$ on the total space of $\Omega^1_{C\times S/ S}\otimes \mathfrak{g}$ by affine linear transformations (at least for points in the small \'etale site of $C \times S$). We can use $\phi$ to form the associated affine bundle $\mathcal{P} \times^G \left( \Omega^1_{C\times S/ S}\otimes \mathfrak{g} \right)$. This can be constructed by choosing an \'etale trivialization of $\mathcal{P}$, and therefore we only need to know $\phi$ on the small \'etale site of $C \times S$ in order to describe the descent data. A $G$-connection on $\mathcal{P}$ relative to $S$ is defined to be a section of the affine morphism $\mathcal{P} \times^G \left( \Omega^1_{C\times S/ S}\otimes \mathfrak{g} \right) \to C \times S$. By construction the bundle of $G$-connections $\mathcal{P} \times^G \left( \Omega^1_{C\times S/ S}\otimes \mathfrak{g} \right)$ is an \'etale torsor for the $C \times S$-sheaf $Ad \,  \mathcal{P} \otimes \Omega^1_{C \times S/ S}$, which is viewed as an abelian sheaf under addition. This torsor represents a cohomology class $\gamma_{\mathcal{P}} \in H^1\left( C \times S, \, Ad \,  \mathcal{P} \otimes \Omega^1_{C \times S/ S}\right)$ called the Atiyah class of $\mathcal{P}$ relative to $S$. Here $H^1\left( C \times S, \, Ad \,  \mathcal{P} \otimes \Omega^1_{C \times S/ S}\right)$ is the ordinary sheaf cohomology group in the Zariski site. It coincides with the corresponding \'etale cohomology group because $Ad \,  \mathcal{P} \otimes \Omega^1_{C \times S/ S}$ is quasicoherent \cite[\href{https://stacks.math.columbia.edu/tag/03P2}{Tag 03P2}]{stacks-project}.
	
	We can think of the Atiyah class $\gamma_{\mathcal{P}} \in H^1\left(C \times S, \, Ad \, \mathcal{P} \otimes \Omega^1_{C \times S / S}\right)$ as an element of $\text{Ext}^1\left( \mathcal{O}_{C \times S}, \,\text{Ad} \, \mathcal{P} \otimes \Omega^1_{C \times S/S}\right)$. It corresponds to an extension of $\mathcal{O}_{C \times S}$-modules
	\[ 0 \longrightarrow Ad \, \mathcal{P} \otimes \Omega^1_{C \times S / S} \longrightarrow At(\mathcal{P}) \longrightarrow \mathcal{O}_{C \times S} \longrightarrow 0  \]
	We call $At(\mathcal{P})$ the Atiyah bundle. The short exact sequence above will be referred to as the Atiyah sequence. By definition a $G$-connection on the $G$-bundle $\mathcal{P}$ is the same as a $\mathcal{O}_{C \times S}$-linear splitting of the Atiyah sequence. Such splitting exists if and only if the Atiyah class $\gamma_{\mathcal{P}}$ vanishes.
	\begin{remark}
		In the literature (e.g. \cite{biswas-atiyahbundle}) the Atiyah sequence differs from the one we have defined. It is usually the twist of the sequence above by $\left(\Omega^1_{C \times S/ S}\right)^{\vee}$.
	\end{remark} 

	\begin{example} \label{example: atiyah class line bundle}
	Suppose that $S = \text{Spec} \, k$ and $G = \mathbb{G}_m$. A $\mathbb{G}_m$-torsor is the same thing as a line bundle on $C$. Let $\mathcal{L}$ be such a line bundle. By definition  $\gamma_{\mathcal{L}} \in H^{1}\left( C, \, Ad \, \mathcal{L} \otimes \Omega_{C / k}^{1} \right)$. Since $\mathbb{G}_m$ is abelian, there is a canonical identification $Ad \, \mathcal{L} \cong \mathcal{O}_C \otimes \mathfrak{gl}_1$. Hence $\gamma_{\mathcal{L}} \in H^{1}\left( C, \, \Omega_{C / k}^{1} \otimes \mathfrak{gl}_1 \right)$. Serre duality yields an isomorphism $H^1\left(C , \, \Omega_{C / k}^{1} \otimes \mathfrak{gl}_1 \right) \cong \left(H^0\left(C, \,\mathcal{O}_C\right) \otimes \mathfrak{gl}_1 \right)^{\vee}$.
	
	Since $C$ is geometrically irreducible and proper, there is a canonical isomorphism of $k$-algebras $H^0\left( C, \, \mathcal{O}_C \right) \cong k$. Under this identification we have $\gamma_{\mathcal{L}} \in \mathfrak{gl}_1^{\vee}$. Finally, there is a canonical isomorphism $\mathfrak{gl}_1 \cong k$ obtained via the natural inclusion $\mathbb{G}_m \hookrightarrow \mathbb{A}^{1}_{k}$. So we can view $\gamma_{\mathcal{L}}$ as an element of  $k$.
	\end{example}
	\begin{lemma} \label{lemma: atiyah class line bundle}
	Let $\mathcal{L}$ be as in Example \ref{example: atiyah class line bundle}. Under the identification above, we have $\gamma_{\mathcal{L}} = -\text{deg} \,  \mathcal{L}$.
	\end{lemma}
	\begin{proof}
	This follows from a Cech cohomology computation similar to the one in \cite{atiyah-connections} Section 3.
	\end{proof}
	
Part (i) in the following proposition is the simpler of the two directions in a theorem of Weil. See \cite{biwas-atiyah-weil} Section 7 for an exposition. We provide a proof here for completeness.
\begin{lemma} \label{lemma: weil theorem}
	Suppose that the characteristic of $k$ is $0$. Let $\mathcal{E}$ be a vector bundle on $C$ equipped with a connection. 
	\begin{enumerate}[(i)]
	    \item Every direct summand $\mathcal{F}$ of $\mathcal{E}$ satisfies $\text{deg}\, \mathcal{F} =0$. In particular $\mathcal{E}$ itself has degree $0$.
	    \item If $\mathcal{F}$ is a subbundle of $\mathcal{E}$ that is stable under the connection, then $\text{deg} \, \mathcal{F} = 0$.
	\end{enumerate}
\end{lemma}
\begin{proof}
\quad \newline
\noindent (i) The connection on $\mathcal{E}$ corresponds to a morphism of abelian sheaves $\nabla: \mathcal{E} \rightarrow \mathcal{E} \otimes \Omega^1_{C/k}$ that satisfies the Leibniz rule. For any direct summand $\mathcal{F}$ of $\mathcal{E}$, we have that $\nabla$ also induces a connection on $\mathcal{F}$. Using the functoriality of connections for the determinant morphism $\text{det}: \text{GL}_{(\text{rank} \, \mathcal{F})} \rightarrow \mathbb{G}_m$, we conclude that the determinant bundle $\text{det}\, \mathcal{F}$ also admits a connection. This means that $\gamma_{\text{det}\, \mathcal{F}} =0$. By Lemma \ref{lemma: atiyah class line bundle}, this is equivalent to $-\text{deg}\, \mathcal{F} = 0$ when viewed as an element of $k$. Since the characteristic of $k$ is $0$, it follows that $\text{deg} \, \mathcal{F} = 0$.
\medskip

\noindent (ii) Suppose that $\mathcal{F} \subset \mathcal{E}$ is stable under the connection $\theta$. Then, we can restrict $\theta|_{\mathcal{F}}$ in order to obtain a connection defined on $\mathcal{F}$. Now part (i) shows that $\text{deg}\, \mathcal{F} = 0$.
\end{proof}
\end{subsection}
\begin{subsection}{Functoriality for flat $G$-connections} \label{subsection: functoriality regular connections}
Connections are contravariant with respect to morphisms of base schemes. To see why this is the case, let $f: T \longrightarrow S$ be a morphism of $k$-schemes. Let $\mathcal{P}$ be a $G$-bundle on $C \times S$. By definition the bundle of $G$-connections $(Id_C \times f)^{*} \mathcal{P} \times^G \left( \Omega^1_{C\times T/ T}\otimes \mathfrak{g} \right)$ is the pullback of the $Ad \,  \mathcal{P} \otimes \Omega^1_{C \times S/ S}$-torsor $\mathcal{P} \times^G \left( \Omega^1_{C\times S/ S}\otimes \mathfrak{g} \right)$. This means that the Atiyah class behaves well under pullbacks, meaning $\gamma_{(Id_C \times f)^{*} \mathcal{P}} = (Id_C \times f)^{*} \gamma_{\mathcal{P}}$. Here we write $(Id_C \times f)^{*}$ on the right-hand side to denote the natural map on sheaf cohomology groups 
\begin{gather*}
H^1\left(C \times S, \, Ad \, \mathcal{P} \otimes \Omega^1_{C \times S / S}\right) \, \longrightarrow \, H^1\left(C \times T, \,(Id_c \times f)^{*} \left( Ad \, \mathcal{P} \otimes \Omega^1_{C \times S / S} \right) \,\right) \, \xlongequal{\; \; } \, H^1\left(C \times T, \, Ad \, (Id_C \times f)^{*}\mathcal{P} \otimes \Omega^1_{C \times T / T}\right)
\end{gather*}
As a consequence, the Atiyah sequence for $(Id_C \times f)^{*} \mathcal{P}$ is the $(Id_C \times f)$-pullback of the Atiyah sequence for $\mathcal{P}$. We can therefore pullback connections by interpreting them as splittings of the Atiyah sequence. For a $G$-connection $\theta$ on $\mathcal{P}$, we denote by $f^{*}\theta$ the corresponding connection on $(Id_C \times f)^{*} \mathcal{P}$.
    
On the other hand, connections are covariant with respect to morphisms of structure groups. Fix a $k$-scheme $S$. Let $H$ be a linear algebraic group over $k$ with Lie algebra $\mathfrak{h}$. Let $\varphi: G \longrightarrow H$ be a homomorphism of algebraic groups. For any $G$-bundle $\mathcal{P}$ over $C \times S$, there exists an associated $H$-bundle $\varphi_{*} \mathcal{P} \vcentcolon = \mathcal{P} \times^G H$ obtained via extension of structure group. Observe that there is an induced morphism of tangent spaces at the identity $\text{Lie}(\varphi) : \mathfrak{g} \longrightarrow \mathfrak{h}$. By \'etale descent for morphisms of quasicoherent sheaves \cite[\href{https://stacks.math.columbia.edu/tag/023T}{Tag 023T}]{stacks-project}, this yields a vector bundle morphism $Ad \, \varphi: Ad\, \mathcal{P} \longrightarrow Ad \, \varphi_{*} \mathcal{P}$. This induces the following map of sheaf cohomology groups.
 \[\varphi^1_{*}: H^1\left( C \times S, \, Ad\, \mathcal{P} \otimes\Omega^1_{C\times S/ S}\right) \longrightarrow H^1\left( C \times S, \, Ad\, \varphi_{*} \mathcal{P} \otimes\Omega^1_{C\times S/ S}\right)\]
 By construction $\varphi_{*}^1 \left(\gamma_{\mathcal{P}} \right) = \gamma_{\varphi_{*} \mathcal{P}}$.

Recall that the functoriality of $\text{Ext}^1(\mathcal{O}_{C\times S}, -)$ can be described concretely in terms of pushouts \cite[\href{https://stacks.math.columbia.edu/tag/010I}{Tag 010I}]{stacks-project}. This description shows that there is a canonical commutative diagram of short exact sequences between the extension defined by $\gamma_{\mathcal{P}}$ and the extension defined by $\varphi_*^{1}(\gamma_{\mathcal{P}})$. Since $\varphi_{*}^1 \left(\gamma_{\mathcal{P}} \right) = \gamma_{\varphi_{*} \mathcal{P}}$, there is a commutative diagram of Atiyah sequences
\begin{figure}[H]
\centering
\begin{tikzcd}
    0  \ar[r]& \text{Ad} \, \mathcal{P} \otimes \Omega^1_{C \times S/S} \ar[d, "Ad \, \varphi \, \otimes \, id"] \ar[r] & At(\mathcal{P}) \ar[r] \ar[d, "At(\varphi)"]  & \mathcal{O}_{C \times S} \ar[r] \ar[d, symbol =  \xlongequal{}] & 0\\
    0 \ar[r] & \text{Ad} \, \varphi_{*} \mathcal{P} \otimes \Omega^1_{C \times S/S} \ar[r] & At(\varphi_{*}\mathcal{P}) \ar[r] & \mathcal{O}_{C \times S} \ar[r] & 0
\end{tikzcd}
\end{figure}
\vspace{-0.5cm}
We denote the middle vertical morphism by $\text{At}(\varphi)$. Let $\theta$ be a $G$-connection on $\mathcal{P}$, viewed as a splitting of the top row. We compose $\theta$ with $At(\varphi)$ in order to define a $H$-connection $\varphi_{*} \theta \vcentcolon = At(\varphi) \circ \theta$ on the $H$-bundle $\varphi_{*} \mathcal{P}$.
\end{subsection}

\begin{subsection}{The stack of flat $G$-connections}
The contravariant functoriality for connections with respect maps of base schemes implies that the $G$-connections form a pseudofunctor into groupoids. This allows to define the moduli stack of flat $G$-bundles over $C$.
\begin{defn}
The moduli stack $\text{Conn}_{G}(C)$ of flat $G$-bundles over $C$ is the pseudofunctor from $k$-schemes to groupoids defined as follows. For every $k$-scheme $S$, we define
\[\text{Conn}_{G}(C) \, (S) \vcentcolon = \; \left\{ \begin{matrix} \text{groupoid of $G$-torsors} \; \mathcal{P} \text{ over } C \times S \\ $+$\\ \text{ a $G$-connection $\theta$ on $\mathcal{P}$ relative to $S$} \end{matrix} \right\} \]
The isomorphisms are required to be compatible with the connections.
\end{defn}

\begin{remark}
We have chosen to use the language of pseudofunctors \cite[\href{https://stacks.math.columbia.edu/tag/003N}{Tag 003N}]{stacks-project} instead of the technically cleaner formulation of stacks in terms of categories fibered in groupoids \cite[\href{https://stacks.math.columbia.edu/tag/003S}{Tag 003S}]{stacks-project}. The reason for this is that we believe the definitions are more intuitive for our purposes in this paper. We note however that one can equivalently formulate everything in terms of fibered categories with a cleavage by \cite[Part 1. \S3.1.3]{fga-explained}.
\end{remark}

There is a morphism of pseudofunctors $Forget: \text{Conn}_{G}(C) \longrightarrow \text{Bun}_G(C)$ given by forgetting the connection.
\begin{prop} \label{prop: moduli of connections affine bundle}
The morphism $Forget: \text{Conn}_{G}(C) \longrightarrow \text{Bun}_G(C)$ is schematic, affine and of finite type. In particular $\text{Conn}_G(C)$ is an algebraic stack that is locally of finite type over $k$.
\end{prop}
\begin{proof}
Choose a $k$-scheme $S$ and a morphism $S \to \text{Bun}_{G}(C)$ corresponding to a $G$-bundle $\mathcal{P}$ on $C \times S$. The Atiyah sequence
\[0 \longrightarrow Ad \, \mathcal{P} \otimes \Omega^1_{C \times S / S} \longrightarrow At(\mathcal{P}) \longrightarrow \mathcal{O}_{C \times S} \longrightarrow 0 \]
corresponds to a torsor for the vector bundle groups scheme $Ad \, \mathcal{P} \otimes \Omega^1_{C \times S / S}$ on $C \times S$. Let us denote by $X \to C \times S$ the total space of this torsor. Note that the morphism $X \to C\times S$ is schematic, affine and of finite presentation. This follows from \cite[\href{https://stacks.math.columbia.edu/tag/0245}{Tag 0245}]{stacks-project}+ \cite[\href{https://stacks.math.columbia.edu/tag/02L5}{Tag 02L5}]{stacks-project}+ \cite[\href{https://stacks.math.columbia.edu/tag/02L0}{Tag 02L0}]{stacks-project}, since \'etale locally on $C \times S$ we have that $X$ is isomorphic to the total space of the pullback of the vector bundle $Ad \, \mathcal{P} \otimes \Omega^1_{C \times S / S}$, which is relatively affine and of finite presentation. By definition, sections $s: C \times S \to X$ are in natural correspondence with splittings of the Atiyah sequence. We can use this to describe the fiber product $\text{Conn}_{G}(C) \times_{\text{Bun}_{G}(C)} S$. Indeed, it follows that for any scheme $T \to S$ we have
\begin{align*}
    (\text{Conn}_{G}(C) \times_{\text{Bun}_{G}(C)} S)(T) & = \left\{ \; \text{sections of the base-changed torsor $X \times_{S} T \to C\times T$}\; \right\}\\
    & = \Gamma_{C\times S/S}(X)(T)
\end{align*}
By \Cref{lemma: representability of sections}, the functor $\Gamma_{C\times S/S}(X)$ is represented by a scheme that is relatively affine and of finite type over $S$. It follows that $\text{Conn}_{G}(C) \times_{\text{Bun}_{G}(C)} \to S$ is affine and of finite type, as desired.
\end{proof}

We now use the $Forget$ map to prove that the stack $\text{Conn}_G(C)$ is quasicompact whenever $\text{char} \, k = 0$.
\begin{prop} \label{prop: quasicompactness regular connections}
Suppose that the characteristic of $k$ is $0$. The algebraic stack $\text{Conn}_G(C)$ is of finite type over $k$.
\end{prop}
\begin{proof}
By Proposition \ref{prop: moduli of connections affine bundle}, it suffices to show that that the map $Forget: \text{Conn}_{G}(C) \longrightarrow \text{Bun}_G(C)$ factors through an open substack $W \subset \text{Bun}_G(C)$ that is of finite type over $k$. We shall prove this.
	
Let $U$ be the unipotent radical of $G$. Denote by $\rho_{G}: G \rightarrow G/U$ the correspoding quotient morphism. Choose a faithful representation $\rho: G/U \longrightarrow \text{GL}_n$ of the reductive group $G/U$. For any $G$-bundle $\mathcal{P}$ we can form the associated $\text{GL}_n$-bundle $(\rho \circ \rho_{G})_{*} \mathcal{P}$ via extension of structure group. This construction defines a morphism $(\rho \circ \rho_{G})_{*}: \text{Bun}_G(C) \longrightarrow \text{Bun}_{\text{GL}_n}(C)$. For any $G$-connection $\theta$ on $\mathcal{P}$ we get an induced $\text{GL}_n$-connection $(\rho \circ \rho_{G})_{*}\theta$ on $\rho_{*} \mathcal{P}$. There is a commutative diagram
\[\xymatrix{
		\text{Conn}_{G}(C) \ar[r] \ar[d]^{Forget} & \text{Conn}_{G/U}(C) \ar[r] \ar[d]^{Forget} & \text{Conn}_{\text{GL}_n}(C) \ar[d]^{Forget}\\ \text{Bun}_{G}(C) \ar[r]^{(\rho_{G})_*} &
		\text{Bun}_{G/U}(C) \ar[r]^{\rho_{*}} & \text{Bun}_{\text{GL}_n}(C) } \]
Propositions \ref{prop: quasicompactness associated bundle map} and \ref{prop: quasicompactness associated bundle morphism unipotent} show that the two horizontal morphisms at the bottom of the diagram above are of finite type. Therefore the composition $(\rho \circ \rho_{G})_{*}$ is of finite type. The claim will follow if we can show that the forgetful map $Forget: \text{Conn}_{\text{GL}_n}(C) \longrightarrow \text{Bun}_{\text{GL}_n}(C)$ factors through an open substack $W \subset \text{Bun}_{\text{GL}_n}(C)$ of finite type over $k$.

For this purpose, we will use the Harder-Narasimhan stratification of $\text{Bun}_{\text{GL}_n}$ \cite{schieder-hnstratification}. For any rational cocharacter $\lambda \in \mathbb{Q}^n$, let $\text{Bun}_{\text{GL}_n}^{\leq \lambda}(C)$ be the quasicompact open substack of $\text{Bun}_{GL_n}(C)$ consisting of the strata with Harder-Narasimhan polygon smaller than $\lambda$. We will show that $Forget: \text{Conn}_{\text{GL}_n}(C) \longrightarrow \text{Bun}_{\text{GL}_n}(C)$ factors through a finite union $\bigcup_{j}\text{Bun}_{\text{GL}_n}^{\leq \lambda_j}(C)$ for some $\lambda_j$. This can be checked at the level of field-valued points. Hence it suffices to check that the set of Harder-Narasimhan polygons of geometric points $\overline{t} \longrightarrow \text{Bun}_{\text{GL}_n}(C)$ with nonempty fiber $ \text{Conn}_{\text{GL}_n}(C)_{\overline{t}}$ is a finite set. 

Let $\mu = (\mu_1 \geq \mu_2 \geq ... \geq \mu_n)$ be a tuple of rational numbers in $\frac{1}{n!}\mathbb{Z}^n$. Let $\mathcal{E}: \overline{t} \longrightarrow \text{Bun}_{\text{GL}_n}(C)$ be a vector bundle on $C\times \overline{t}$, where $\overline{t}$ is $\text{Spec}$ of an algebrically closed field. Suppose that the Harder-Narasimhan polygon of $\mathcal{E}$ is given by $\mu$. Assume that the fiber $ \text{Conn}_{\text{GL}_n}(C)_{\overline{t}}$ is nonempty. This means that $\mathcal{E}$ admits a connection. We claim that $\mu$ satisfies the following two conditions
\begin{enumerate}[(a)]
    \item $\sum_{j =1}^n \mu_j = 0$.
    \item $\mu_{j} - \mu_{j+1} \leq \text{max}\{2g-2, 0\}$ for all $1 \leq j \leq n-1$.
\end{enumerate}
The claim implies that $-\left(\text{max}(\{g-1, 0\}\right) \cdot (n-1) \leq \mu_n \leq \mu_1 \leq \left(\text{max}\{g-1, 0\}\right)\cdot (n-1)$. This would show that there are finitely many possibilities for $\mu \in \frac{1}{n!}\mathbb{Z}^n$. We are left to proving the claim.
\medskip

\noindent (a) By Lemma \ref{lemma: weil theorem}, $\text{deg}\, \mathcal{E} = 0$. This is equivalent to the condition (a) in the claim above.
\medskip

\noindent (b) For the sake of contradiction, assume that $\mu_{j+1} - \mu_{j}> \text{max}\{2g-2, 0\}$ for some $1 \leq j \leq n-1$. Suppose that $\mathcal{E}$ has Harder-Narasimhan filtration
\[0 \subset \mathcal{F}_1 \subset \mathcal{F}_2 \subset ... \subset \mathcal{F}_l = \mathcal{E}   \]
There is an index $1 \leq k \leq l-1$ such that $\mu\left( \mathcal{F}_{k} \, / \, \mathcal{F}_{k-1} \right) = \mu_j$. We have a short exact sequence
\[0 \longrightarrow \mathcal{F}_k \longrightarrow \mathcal{E} \longrightarrow \mathcal{E} \, / \, \mathcal{F}_k \longrightarrow 0  \]
An application of Serre duality plus the definition of semistability and the hypothesis $\mu_{j+1} - \mu_{j}> \text{max}(2g-2, 0)$ implies that $\text{Ext}^1\left(\mathcal{E}\, / \, \mathcal{F}_k \,, \, \mathcal{F}_k \right) = 0$. We conclude that $\mathcal{E} = \mathcal{F}_k \, \oplus \, \mathcal{E}\, / \, \mathcal{F}_k$. But by assumption we must have 
\[\text{deg} \, \mathcal{F}_k \, >  \, \text{deg} \,\mathcal{E}  \, > \, \text{deg} \,\mathcal{E} \, / \, \mathcal{F}_k\]
Hence $\text{deg} \, \mathcal{F}_k \, > \, 0$. This contradicts Lemma \ref{lemma: weil theorem} (i), because $\mathcal{F}_k$ is a direct summand of $\mathcal{E}$.
\end{proof}

We record another proof that $\text{Conn}_{\text{GL}_{n}}(C)$ is of finite type as a consequence of Simpson's work \cite{Simpson-repnI}. The stack $\text{Conn}_{\text{GL}_{n}}(C)$ classifies rank $n$ vector bundles equipped with a connection. These are $\Lambda$-modules of pure dimension $1$ on $C$, where $\Lambda$ is the sheaf of differential operators on $C$ (using the notation in \cite{Simpson-repnI}[pg. 85-86]). 
If $\mathcal{E}$ is a vector bundle with a connection, then $\text{deg}(\mathcal{E}) = 0$ by Lemma \ref{lemma: weil theorem}. If furthermore $\mathcal{E}$ has rank $n$, then we can use Riemann-Roch to see that the Hilbert polynomial of $\mathcal{E}$ is given by
\[p(\mathcal{E}, x) = \left(n \cdot \text{deg} \, \mathcal{O}_{C}(1)\right) x + n \cdot(1-g) \]
In particular the slope of $\mathcal{E}$ with respect to $\mathcal{O}_{C}(1)$ is given by $\widehat{\mu}(\mathcal{E}) = \frac{1-g}{\text{deg}\, \mathcal{O}_{C}(1)}$. Suppose that $\mathcal{G} \subset \mathcal{E}$ is a $\Lambda$-submodule. This means that $\mathcal{G}$ is a subbundle that is stable under the connection. Lemma \ref{lemma: weil theorem} (ii) implies that $\text{deg} \,  \mathcal{G} =0$. The same Riemann-Roch computation can be applied to see that $\widehat{\mu}(\mathcal{G}) = \frac{1-g}{\text{deg}\, \mathcal{O}_{C}(1)}$. This shows that $\mathcal{E}$ is a $\mu$-semistable $\Lambda$-module in the sense of \cite{Simpson-repnI}[\S3]. It follows that $\text{Conn}_{\text{GL}_n}(C)$ classifies $\mu$-semistable $\Lambda$-modules with fixed Hilbert polynomial $P(x) = \left(n \cdot \text{deg} \, \mathcal{O}_{C}(1)\right) x + n \cdot(1-g)$. We can therefore apply \cite{Simpson-repnI}[Cor. 3.4] to conclude that the moduli stack $\text{Conn}_{\text{GL}_{n}}(C)$ is quasicompact.

In fact \cite{Simpson-repnI}[Lemma 3.3] provides a bound on the Harder-Narasimhan polygons of points in the image of the morphism $\text{Forget}: \text{Conn}_{\text{GL}_n}(C) \rightarrow \text{Bun}_{\text{GL}_n}(C)$. We assume that the ample line bundle $\mathcal{O}_{C}(1)$ has degree $1$. Using the notation of \cite{Simpson-repnI}[Lemma 3.3], we have that $\text{Gr}_{1}(\Lambda)$ is the tangent bundle of $C$. A Riemann-Roch computation shows that we can take $m = 4g-3$ in \cite{Simpson-repnI}[Lemma 3.3]. Let $\mathcal{E} \in \text{Forget}\left(\text{Conn}_{\text{GL}_n}(C)\right)$ be a vector bundle with Harder-Narasimhan type $\mu= (\mu_1 \geq \mu_2 \geq ... \geq \mu_n)$. We have $\mu_1 = \mu(\mathcal{F}_1) = \widehat{\mu}(\mathcal{F}_1) + (1-g)$, where $\mathcal{F}_1$ is the first step in the Harder-Narasimhan filtration. \cite{Simpson-repnI}[Lemma 3.3] shows that $\mu_1 \leq \left(\text{max}\{4g-3, 0\}\right)\cdot n$. Using this inequality and the fact that $\sum_{j=1}^{n} \mu_j =0$, we get a bound on the Harder-Narasimhan polygon of $\mathcal{E}$. Hence there are finitely many possibilities for $\mu$. This bound on the Harder-Narasimhan type is weaker than the one obtained in the proof of Proposition \ref{prop: quasicompactness regular connections}.
\end{subsection}
\end{section}
\begin{section}{The moduli stack of logarithmic $G$-connections}
\begin{subsection}{Logarithmic $G$-connections and their residues}
We start this section by recalling the necessary preliminaries on logarithmic connections. Our initial exposition is adapted from the holomorphic case considered in \cite{Biswas_2017}[2.2, 2.3]. 

Let $D = \sum_{i \in I} x_i$ be the reduced divisor of $C$ supported at the $x_i$'s. There is a canonical inclusion of sheaves $\mathcal{O}_{C \times S}(-D) \hookrightarrow \mathcal{O}_{C\times S}$ induced by $x_i$. The logarithmic Atiyah sequence $\text{At}^D(\mathcal{P})$ with poles at $D$ is defined by the following pullback diagram.
\begin{figure}[H]
\centering
\begin{tikzcd}
     At^D(\mathcal{P}) \ar[r] \ar[d, symbol = \hookrightarrow]  & \mathcal{O}_{C \times S}(-D) \ar[d, symbol = \hookrightarrow] \\
     At(\mathcal{P}) \ar[r] & \mathcal{O}_{C \times S}
\end{tikzcd}
\end{figure}
\vspace{-0.5cm}
By definition, there is a commutative diagram of short exact sequences
\begin{figure}[H]
\centering
\begin{tikzcd}
    0  \ar[r]& \text{Ad} \, \mathcal{P} \otimes \Omega^1_{C \times S/S} \ar[d, symbol = \xlongequal{}] \ar[r] & At^D(\mathcal{P}) \ar[r] \ar[d, symbol = \hookrightarrow]  & \mathcal{O}_{C \times S}(-D) \ar[r] \ar[d, symbol = \hookrightarrow] & 0\\
    0 \ar[r] & \text{Ad} \, \mathcal{P} \otimes \Omega^1_{C \times S/S} \ar[r] & At(\mathcal{P}) \ar[r] & \mathcal{O}_{C \times S} \ar[r] & 0
\end{tikzcd}
\end{figure}
\vspace{-0.5cm}
The top row is called the logarithmic Atiyah sequence with poles at $D$. A splitting of this sequence is called a logarithmic connection with poles at $D$. By construction it makes sense to ask whether a logarithmic connection $\theta$ of $\mathcal{P}$ extends to a $G$-connection of $\mathcal{P}$. This just means that the composition $ \mathcal{O}_{C \times S}(-D) \xrightarrow{\theta} \text{At}^D(\mathcal{P}) \rightarrow \text{At}(\mathcal{P})$ (uniquely) extends to $\mathcal{O}_{C \times S}$.
\begin{example}
Suppose that $S = \text{Spec} \, k$ and $G = \mathbb{G}_m$. Let $\mathcal{L}$ be a line bundle on $C$. If the divisor $D$ is nonzero, the logarithmic Atiyah sequence for $\mathcal{L}$ splits. Therefore all line bundles admit logarithmic connections.
\end{example}
We tensor the logarithmic Atiyah sequence with $\mathcal{O}_{C \times S}(D)$ to obtain the following diagram
\begin{figure}[H]
\centering
\begin{tikzcd}
    0  \ar[r]& \text{Ad} \, \mathcal{P} \otimes \Omega^1_{C \times S/S}(D) \ar[r] & At^D(\mathcal{P})(D) \ar[r]  & \mathcal{O}_{C \times S} \ar[r] & 0
\end{tikzcd}
\caption{Diagram 1}
\label{diagram: 1}
\end{figure}
\vspace{-0.5cm}
Let $i \in I$. Base-changing Diagram \ref{diagram: 1} to the closed subscheme $x_i \times S$ we get a diagram with exact rows and columns
\begin{figure}[H]
\centering
\begin{adjustbox}{width = \textwidth}
\begin{tikzcd}[cramped]
& 0 \ar[d] & 0 \ar[d] & 0 \ar[d] & \\
     0 \ar[r] & 0 \ar[r] \ar[d] & \left(At^D(\mathcal{P}) \, / \, At(\mathcal{P})\right) (D) \, |_{x_i \times S} \ar[r, "\psi_i"] \ar[d] & \mathcal{O}_{x_i \times S} \ar[r] \ar[d, "Id"] & 0 \\
0 \ar[r] &  Ad(\mathcal{P}) \otimes \Omega^1_{C \times S/ S}(D) \,|_{x_i \times S} \ar[r] \ar[d, "Id"] & At^D(\mathcal{P})(D)|_{x_i \times S}  \ar[r] \ar[d]& \mathcal{O}_{x_i \times S} \ar[r] \ar[d, "0"] & 0\\
0 \ar[r] & Ad(\mathcal{P}) \otimes \Omega^1_{C \times S/ S}(D) \,|_{x_i \times S} \ar[r] \ar[d] & At(\mathcal{P})(D)|_{x_i \times S} \ar[d] \ar[r] & \mathcal{O}_{x_i \times S} \ar[r] \ar[d, "Id"] & 0\\
0 \ar[r] & 0 \ar[r] \ar[d] & \left(At^D(\mathcal{P}) \, / \, At(\mathcal{P})\right)(D) \, |_{x_i \times S} \ar[r, "\psi_i"] \ar[d] & \mathcal{O}_{x_i \times S} \ar[r] \ar[d] & 0\\
& 0 & 0 & 0 &
\end{tikzcd}
\end{adjustbox}
\end{figure}
\vspace{-0.5cm}
The first two rows yield a canonical splitting of the base-change of Diagram 1  to $x_i \times S$
\begin{figure}[H]
\centering
\begin{adjustbox}{width = \textwidth}
\begin{tikzcd}
    0  \ar[r]& \text{Ad} \, \mathcal{P} \otimes \Omega^1_{C \times S/S}(D) \, |_{x_i \times S} \ar[d, symbol = \xlongequal{}] \ar[r] & At^D(\mathcal{P})(D)|_{x_i \times S} \ar[r] \ar[d, symbol = \xrightarrow{\sim}]  & \mathcal{O}_{x_i \times S} \ar[r] \ar[d, symbol = \xlongequal{}] & 0\\
    0 \ar[r] & \text{Ad} \, \mathcal{P} \otimes \Omega^1_{C \times S/S}(D) \, |_{x_i \times S} \ar[r] & \text{Ad} \, \mathcal{P} \otimes \Omega^1_{C \times S/S}(D) \, |_{x_i \times S} \oplus \mathcal{O}_{x_i \times S}  \ar[r] & \mathcal{O}_{x_i \times S} \ar[r] & 0
\end{tikzcd}
\end{adjustbox}
\caption{Diagram 2}
\label{diagram: 2}
\end{figure}
\vspace{-0.5cm}
Consider the fiber $\Omega^1_{C/ k}(D)|_{x_i}$. It is a free $\mathcal{O}_{x_i}$-module of rank $1$. There is a canonical generating element given as follows. Choose a uniformizer $z_i$ of the discrete valuation local ring $\mathcal{O}_{C,x_i}$. Then, the meromorphic $1$-form $\frac{d z_i}{z_i}$ is a generator of $\Omega^1_{C/k}(D)|_{\mathcal{O}_{C, x_i}}$, and its restriction to the fiber at $x_i$ does not depend on the choice of uniformizer. This induces a canonical trivialization $\Omega^1_{C/ k}(D)|_{x_i} \cong \mathcal{O}_{x_i}$. This in turn yields an isomorphism $\Omega^1_{C \times S/ S}(D)|_{x_i \times S} \cong \mathcal{O}_{x_i \times S}$ after base-changing to $x_i \times S$. This canonical isomorphism $\Omega^1_{C \times S/ S}(D)|_{x_i \times S} \cong \mathcal{O}_{x_i \times S}$ and the splitting in Diagram \ref{diagram: 2} can be used to identify the $x_i \times S$ fiber of Diagram \ref{diagram: 1} with the following split short exact sequence
\begin{figure}[H]
\centering
\begin{tikzcd}
    0 \ar[r] & \text{Ad} \left(\mathcal{P}|_{x_i \times S}\right) \ar[r] & \text{Ad} \left(\mathcal{P}|_{x_i \times S}\right) \oplus \mathcal{O}_{x_i \times S}  \ar[r] & \mathcal{O}_{x_i \times S} \ar[r] & 0
\end{tikzcd}
\caption{Diagram 3}
\label{diagram: 3}
\end{figure}
\vspace{-0.5cm}
Let $\theta$ be a logarithmic connection on $\mathcal{P}$ with poles at $D$, viewed as a splitting of the logarithmic Atiyah sequence. We can tensor with $\mathcal{O}_{C \times S}(D)$ and pass to the $x_i \times S$ fiber to obtain a splitting of Diagram \ref{diagram: 3}. Such splittings are in natural correspondence with global sections of the vector bundle $\text{Ad} \left(\mathcal{P}|_{x_i \times S}\right)$. The section corresponding to the splitting induced by $\theta$ is called the residue of $\theta$ at $x_i$, and will be denoted $\text{Res}_{x_i} \theta \in H^0\left(x_i \times S, \, \text{Ad} \left(\mathcal{P}|_{x_i \times S}\right) \, \right)$.

Our next goal is to define a version of the logarithmic Atiyah sequence with prescribed residues. Let $S$ be a $k$-scheme. Let $\mathcal{P}$ be a $G$-bundle on $C \times S$. For each $i \in I$, fix a section $s_i \in H^0\left(x_i \times S, \, Ad(\mathcal{P}|_{x_i \times S}) \, \right)$. Set $W_i$ to be the cokernel of the map $s_i \oplus Id_{\mathcal{O}_{x_i \times S}} : \, \mathcal{O}_{x_i \times S} \longrightarrow Ad(\mathcal{P}|_{x_i \times S}) \oplus \mathcal{O}_{x_i \times S}$. We have seen that there is an isomorphism
\[  At^D(\mathcal{P})(D)|_{x_i \times S} \xrightarrow{\sim} \text{Ad} \left(\mathcal{P}|_{x_i \times S}\right) \oplus \mathcal{O}_{x_i \times S} \]
Define $At^{D,  s_i}(\mathcal{P})$ to be the kernel of the composition
\begin{gather*}
    At^D(\mathcal{P})(D) \xrightarrow{\underset{i \in I}{\oplus} unit} \bigoplus_{i \in I} (q_i \times id_S)_* \left( At^D(\mathcal{P})(D)|_{x_i \times S} \right) \xrightarrow{\sim} \bigoplus_{i \in I} (q_i \times id_S)_* \, \left(Ad(\mathcal{P}|_{x_i \times S}) \oplus \mathcal{O}_{x_i \times S} \right) \twoheadrightarrow \bigoplus_{i \in I} (q_i \times id_S)_* \,  W_i
\end{gather*}
Consider the pullback diagram
\begin{figure}[H]
\centering
\begin{tikzcd}
    0  \ar[r] & \text{Ad} \, \mathcal{P} \otimes \Omega^1_{C \times S/S} \ar[d, symbol = \hookrightarrow] \ar[r] & At^{D, s_i}(\mathcal{P}) \ar[r] \ar[d, symbol = \hookrightarrow]  & \mathcal{O}_{C \times S} \ar[r] \ar[d, symbol = \xlongequal{}] & 0\\
    0 \ar[r] & \text{Ad} \, \mathcal{P} \otimes \Omega^1_{C \times S/S}(D) \ar[r] & At^D(\mathcal{P})(D) \ar[r] & \mathcal{O}_{C \times S} \ar[r] & 0
\end{tikzcd}
\caption{Diagram 4}
\label{diagram: 4}
\end{figure}
\vspace{-0.5cm}
By construction, there is a bijection between logarithmic $G$-connections $\theta$ on $\mathcal{P}$ such that $\text{Res}_{x_i} \theta = s_i$ for all $ i \in I$ and splittings of the top row in Diagram \ref{diagram: 4}. We refer to this short exact sequence as the logarithmic Atiyah sequence with prescribed residues.
\begin{figure}[H]
\centering
\begin{tikzcd}
    0  \ar[r] & \text{Ad} \, \mathcal{P} \otimes \Omega^1_{C \times S/S}  \ar[r] & At^{D,s_i}(\mathcal{P}) \ar[r] & \mathcal{O}_{C \times S} \ar[r] & 0
\end{tikzcd}
\caption{Diagram 5}
\label{diagram: 5}
\end{figure}
\vspace{-0.5cm}
We set $\gamma^{D, s_i}_{\mathcal{P}} \in H^1\left(C \times S, \, Ad \, \mathcal{P} \otimes \Omega^1_{C \times S/ S} \, \right)$ to be the cohomology class associated to the extension in Diagram \ref{diagram: 5}. It is the analogue of the Atiyah class for logarithmic connections with prescribed residues. The $G$-bundle $\mathcal{P}$ admits a logarithmic connection with prescribed residues $s_i$ if and only if $\gamma^{D, s_i}_{\mathcal{P}} = 0$.  

Proposition \ref{prop: atiyah class with prescribed residues} gives an explicit description $\gamma^{D, s_i}_{\mathcal{P}}$. In order to understand this description we need some setup. For each $i \in I$, we have a short exact sequence
\[0 \longrightarrow \mathcal{O}_{C \times S} \longrightarrow \mathcal{O}_{C \times S}(x_i) \xrightarrow{unit} (q_i \times id_S)_* \left(\mathcal{O}_C(x_i) |_{x_i \times S}\right) \longrightarrow 0 \]
We tensor this sequence with $\text{Ad} \, \mathcal{P} \otimes \Omega^1_{C \times S/S}$ in order to obtain
\begin{gather*}
    0 \longrightarrow \text{Ad} \, \mathcal{P} \otimes \Omega^1_{C \times S/S} \longrightarrow \text{Ad} \, \mathcal{P} \otimes \Omega^1_{C \times S/S}(x_i) \longrightarrow (q_i \times id_S)_*\left(\text{Ad} \, \mathcal{P}|_{x_i \times S} \otimes \Omega^1_{C \times S/S}(x_i)|_{x_i \times S}\right) \longrightarrow 0
\end{gather*}
The canonical trivialization $\Omega^1_{C \times S/ S}(x_i)|_{x_i \times S} \cong \mathcal{O}_{x_i \times S}$ described earlier can be used rewrite the short exact sequence above.
\[0 \longrightarrow \text{Ad} \, \mathcal{P} \otimes \Omega^1_{C \times S/S} \longrightarrow \text{Ad} \, \mathcal{P} \otimes \Omega^1_{C \times S/S}(x_i) \longrightarrow (q_i \times id_S)_*\left(\text{Ad} \, \mathcal{P}|_{x_i \times S}\right) \longrightarrow 0 \]
Denote by $\delta_i: H^0\left(x_i \times S, \, \text{Ad} \, \mathcal{P}|_{x_i \times S}\right) \longrightarrow H^1\left(C\times S, \,  \text{Ad} \, \mathcal{P} \otimes \Omega^1_{C \times S/S}\right)$ the corresponding connecting homomorphism of sheaf cohomology groups.

The following proposition is proven in \cite{biswas-criterion-log-prescribed}[Prop. 3.1] in the case when $k = \mathbb{C}$, $G =  \text{GL}_n$ and $S = \text{Spec} \, \mathbb{C}$. The argument found there uses transcendental methods. We give an algebraic proof of the analogous result in the case of an arbitrary base $S$, arbitrary ground field $k$, and any smooth connected linear algebraic group $G$.
\begin{prop} \label{prop: atiyah class with prescribed residues}
Fix a $k$-scheme $S$. Let $\mathcal{P}$ be a $G$-bundle over $C \times S$. Let $s_i \in H^0\left(x_i \times S, \, \text{Ad}(\mathcal{P}|_{x_i \times S}) \, \right)$ be sections as above. Then, we have $\gamma^{D, s_i}_{\mathcal{P}} = \gamma_{\mathcal{P}} - \sum_{i \in I} \delta_i(s_i)$.
\end{prop}
\begin{proof}
Consider the short exact sequence
\[0 \longrightarrow \mathcal{O}_{C \times S} \longrightarrow \mathcal{O}_{C \times S}(D) \xrightarrow{unit} \bigoplus_{i \in I} (q_i \times id_S)_* \left(\mathcal{O}_C(x_i) |_{x_i \times S}\right) \longrightarrow 0\]
By tensoring with $\text{Ad} \, \mathcal{P} \, \otimes \, \Omega_{C \times S/ S}^{1}$ and using the identifications $\Omega^1_{C \times S/ S}(x_i)|_{x_i \times S} \cong \mathcal{O}_{x_i \times S}$ as described above, we obtain a short exact sequence
\begin{gather*}0 \longrightarrow \text{Ad} \, \mathcal{P} \, \otimes \, \Omega^1_{C \times S/S} \, \xrightarrow{\; \;j \; \;} \, \text{Ad} \, \mathcal{P} \, \otimes \, \Omega^1_{C \times S/S}(D) \, \xrightarrow{\; \; u \; \;} \, \bigoplus_{i \in I} (q_i \times id_S)_*\left(\text{Ad} \, \mathcal{P}|_{x_i \times S}\right) \longrightarrow 0
\end{gather*}
The morphisms $j$ and $u$ above are labeled for future use. The construction of the connecting homomorphism in Cech cohomology implies that the connecting homomorphism for this short exact sequence is given by the sum
\[\sum_{i \in I} \delta_i \, : \; \bigoplus_{i \in I} H^0\left(x_i \times S, \, \text{Ad} \, \mathcal{P}|_{x_i \times S}\right) \longrightarrow H^1\left(C\times S, \,  \text{Ad} \, \mathcal{P} \otimes \Omega^1_{C \times S/S}\right)\]
We recall how to describe the extension corresponding to the cohomology class $-\sum_{i \in I} \delta_i(s_i) \in H^1\left(C\times S, \,  \text{Ad} \, \mathcal{P} \otimes \Omega^1_{C \times S/S}\right) = \text{Ext}^1\left( \mathcal{O}_{C \times S}, \,\text{Ad} \, \mathcal{P} \otimes \Omega^1_{C \times S/S}\right)$. Define a global section $v_i : \mathcal{O}_{C \times S} \longrightarrow (q_i \times id_S)_* \left(\text{Ad} \, \mathcal{P}|_{x_i \times S} \right)$ to be the composition
\[ v_i : \; \mathcal{O}_{C \times S} \, \xrightarrow{unit} \, (q_i \times id_S)_* \, \mathcal{O}_{x_i \times S} \,  \xrightarrow{(q_i \times id_S)_* s_i} \, (q_i \times id_S)_* \left(\text{Ad} \, \mathcal{P}|_{x_i \times S} \right)  \]
Let $\mathcal{F}$ be the $\mathcal{O}_{C\times S}$-sheaf given by the pullback diagram
\begin{figure}[H]
\centering
\begin{tikzcd}
    \mathcal{F}  \ar[r, "p_2"] \ar[d, "p_1"] & \mathcal{O}_{C \times S} \ar[d, "\left( \, -v_i \, \right)_{i}"] \\
    \text{Ad} \, \mathcal{P} \otimes \Omega^1_{C \times S/S}(D) \ar[r, "u"] & \bigoplus_{i \in I} (q_i \times id_S)_*\left(\text{Ad} \, \mathcal{P}|_{x_i \times S}\right)
\end{tikzcd}
\end{figure}
\vspace{-0.5cm}
This means that $\mathcal{F}$ is the kernel of the morphism
\[ \text{Ad} \, \mathcal{P} \otimes \Omega_{C \times S / S}^{1}(D)  \, \oplus \, \mathcal{O}_{C \times S} \, \xrightarrow{ \; \; \; u \, + \,  \left( v_i \right)_{i} \; \; \; } \, \bigoplus_{i \in I} \, (q_i \times id_S)_*\left(\text{Ad} \, \mathcal{P}|_{x_i \times S}\right)  \]
Here the maps $p_l$ for $l =1 ,2$ are the natural projections.

The morphism $(j, 0) : \text{Ad} \, \mathcal{P} \otimes \Omega_{C \times S/ S}^{1} \longrightarrow \text{Ad} \, \mathcal{P} \otimes \Omega_{C \times S / S}^{1}(D) \, \oplus \, \mathcal{O}_{C \times S}$ factors through the subsheaf $\mathcal{F} \subset \text{Ad} \, \mathcal{P} \otimes \Omega_{C \times S / S}^{1}(D) \, \oplus \, \mathcal{O}_{C \times S}$. By construction $(j, 0) : \text{Ad} \, \mathcal{P} \otimes \Omega_{C \times S/ S}^{1} \longrightarrow \mathcal{F}$ is a kernel for the morphism $p_2: \mathcal{F} \longrightarrow \mathcal{O}_{C \times S}$. The extension corresponding to the cohomology class $-\sum_{i \in I} \delta_i(s_i)$ is given by
\[0 \longrightarrow \, \text{Ad}\, \mathcal{P} \otimes \Omega_{C \times S / S}^{1} \, \xrightarrow{ \; \; \;\  (j, 0) \; \; \;} \, \mathcal{F} \,  \xrightarrow{\; \; \; \; p_2 \; \; \; \;} \, \mathcal{O}_{C \times S} \, \longrightarrow \, 0 \]
We describe now the extension corresponding to the cohomology class $\gamma_{\mathcal{P}} - \sum_{i \in I} \delta_{i}(s_i)$. Suppose that the Atiyah sequence for $\mathcal{P}$ is given by
	\[ 0 \, \longrightarrow \,  Ad \, \mathcal{P} \otimes \Omega^1_{C \times S / S} \, \xrightarrow{\; \; \; \; b \; \; \; \;} \,  At(\mathcal{P}) \, \xrightarrow{\; \; \; \; p \; \; \; \;} \, \mathcal{O}_{C \times S} \, \longrightarrow \, 0  \]
Define $\mathcal{G}$ to be the kernel of the morphism $p-p_2 \, : \,\text{At}(\mathcal{P}) \oplus \mathcal{F} \longrightarrow \mathcal{O}_{C \times S}$. Consider the map $(b, \, (j,0) \,) : \text{Ad} \, \mathcal{P} \otimes \Omega_{C \times S/ S}^{1} \longrightarrow \text{At}(\mathcal{P}) \oplus \mathcal{F}$. By construction it factors through the subsheaf $\mathcal{G} \subset \text{At}(\mathcal{P}) \oplus \mathcal{F}$. Recall that in the Yoneda group $\text{Ext}^1\left( \mathcal{O}_{C \times S}, \,\text{Ad} \, \mathcal{P} \otimes \Omega^1_{C \times S/S}\right)$ addition is given by the Baer sum \cite[\href{https://stacks.math.columbia.edu/tag/010I}{Tag 010I}]{stacks-project}. The Baer sum $\gamma_{\mathcal{P}} - \sum_{i \in I} \delta_{i}(s_i)$ is given by
\vspace{-0.5cm}
\begin{figure}[H]
\centering
\[ 0 \, \longrightarrow \text{Ad}\, \mathcal{P} \otimes \Omega_{C \times S / S}^{1} \, \xrightarrow{\; \; \; (b, \, 0) \; \; \;} \, \mathcal{G} \, / \, \text{Im} (b, \, (j, 0) \,) \, \xrightarrow{\; \; \; \; p \; \; \; \;} \mathcal{O}_{C \times S} \, \longrightarrow \, 0 \]
\vspace{-0.5cm}
\caption{Diagram 6}
\label{diagram: 6}
\end{figure}
\vspace{-0.5cm}
The proposition amounts to showing that the extension in Diagram \ref{diagram: 6} is isomorphic to the logarithmic Atiyah sequence with prescribed residues in Diagram \ref{diagram: 5}. We will explicitly construct an isomorphism.

There is a canonical inclusion $k : \text{At}(\mathcal{P}) \hookrightarrow \text{At}(\mathcal{P})(D)$. Define the map $w: \text{At}(\mathcal{P}) \, \oplus \,\text{Ad} \, \mathcal{P} \otimes \Omega_{C \times S / S}^{1}(D) \longrightarrow \text{At}(\mathcal{P})(D)$ to be the composition
\[\text{At}(\mathcal{P}) \, \oplus \, \text{Ad} \, \mathcal{P} \otimes \Omega_{C \times S / S}^{1}(D) \, \xrightarrow{(k, \, -b(D))} \text{At}(\mathcal{P})(D) \oplus \text{Ad}(\mathcal{P})(D) \,  \xrightarrow{\; \; + \; \;} \, \text{At}(\mathcal{P})(D)\]
Let $a: \mathcal{G} \longrightarrow \text{At}(\mathcal{P})(D)$ denote the composition
\[ \mathcal{G} \, \hookrightarrow \, \text{At}(\mathcal{P}) \oplus \mathcal{F} \, \xrightarrow{\; \; id \oplus p_1 \; \;} \, \text{At}(\mathcal{P}) \, \oplus \, \text{Ad} \, \mathcal{P} \otimes \Omega_{C \times S / S}^{1}(D) \, \xrightarrow{\; \; \; w \; \; \;} \, \text{At}(\mathcal{P})(D)  \]
By construction $\text{Im}(b, \, (j,0) \,)$ is the kernel of $a$. We are left to show the following two claims.
\begin{enumerate}[(C1)]
    \item The morphism $a$ factors through the subsheaf $\text{At}^{D, s_i}(\mathcal{P}) \subset \text{At}(\mathcal{P})(D)$.
    \item The induced map $a: \mathcal{G} \, / \, \text{Im}(b, \, (j, 0)\, ) \, \longrightarrow \, \text{At}^{D, s_i}(\mathcal{P})$ yields an isomorphism of short exact sequences
\[\begin{tikzcd}
0  \ar[r] & \text{Ad} \, \mathcal{P} \otimes \Omega^1_{C \times S/S} \ar[d, symbol = \xlongequal{}] \ar[r, "{(b, 0)}"] & \mathcal{G} \, / \, \text{Im} (b, \, (j, 0) \,) \ar[r, "p"] \ar[d, "a"]  & \mathcal{O}_{C \times S} \ar[r] \ar[d, symbol = \xlongequal{}] & 0\\
    0 \ar[r] & \text{Ad} \, \mathcal{P} \otimes \Omega^1_{C \times S/S} \ar[r] & \text{At}^{D, s_i}(\mathcal{P}) \ar[r] & \mathcal{O}_{C \times S} \ar[r] & 0
\end{tikzcd}\]
\end{enumerate}
These two claims can be checked at the level of stalks. Let $x$ be a topological point of $C\times S$.

Suppose first that $x$ does not belong to any of the divisors $x_i \times S$ for $i \in I$. Then, the stalk at $x$ of the logarithmic Atiyah sequence in Diagram \ref{diagram: 5} becomes canonically isomorphic to the stalk of the regular Atiyah sequence. Similarly, the stalk at $x$ of Diagram \ref{diagram: 6} becomes canonically isomorphic to the regular Atiyah sequence. This is because all of the operations we have performed above are trivial outside of the divisors $x_i \times S$. Under these identifications, the induced morphism on stalks $a_x$ in (C2) becomes the identity. This concludes the proof when $x$ is not contained in any of the divisors $x_i \times S$.

Assume now that $x$ is contained in $x_i \times S$ for some $i \in I$. The stalks at $x$ of the short exact sequences we have considered are just extensions of finite free modules over the local ring $\mathcal{O}_{C\times S, x}$. It follows that all of these extensions split. Choose a splitting of the $x$-stalk of the regular Atiyah sequence
\begin{figure}[H]
\centering
\begin{adjustbox}{width = \textwidth}
\begin{tikzcd}
    0  \ar[r] & \left(\text{Ad} \, \mathcal{P} \otimes \Omega^1_{C \times S/S}\right)_x \ar[d, symbol = \xlongequal{}] \ar[r, "b_x"] & At(\mathcal{P})_x \ar[r, "p_x"] \ar[d, symbol = \xrightarrow{\sim}]  & \mathcal{O}_{C \times S, x} \ar[r] \ar[d, symbol = \xlongequal{}] & 0\\
    0 \ar[r] & \left(\text{Ad} \, \mathcal{P} \otimes \Omega^1_{C \times S/S}\right)_x \ar[r] & \left(\text{Ad} \, \mathcal{P} \otimes \Omega^1_{C \times S/S}\right)_x \oplus \mathcal{O}_{C \times S, x} \ar[r] & \mathcal{O}_{C \times S, x} \ar[r] & 0
\end{tikzcd}
\end{adjustbox}
\end{figure}
\vspace{-0.5cm}
Such splitting induces an identification $\text{At}^D(\mathcal{P})(D)_x \cong \left(\text{Ad} \, \mathcal{P} \otimes \Omega^1_{C \times S/S}(D)\right)_x \oplus \mathcal{O}_{C \times S, x}$. Furthermore, we can choose the splitting so that this identification is compatible with Diagram \ref{diagram: 2} after base-changing to the residue field of $x$. By choosing a uniformizer $z_i$ for the local ring $\mathcal{O}_{C,x_i}$, we obtain an isomorphism $ \frac{dz_i}{z_i}: \mathcal{O}_{C,x_i} \xrightarrow{\sim} \Omega^1_{C/k}(D)|_{\mathcal{O}_{C, x_i}} $, which induces by base-change an isomorphism $\Omega^1_{C \times S/ S}(D)_x \cong \mathcal{O}_{C \times S, x}$. We can use this to identify the stalks $\text{At}(\mathcal{P})(D)_x \cong \text{Ad}(\mathcal{P})_x \oplus \mathcal{O}(D)_x$. Let $\overline{s}_i \in \text{Ad}(\mathcal{P})_{\kappa(x)}$ be the restriction of  the section $s_i$ to the residue field of $x$. The submodule $\text{At}^{D, s_i}(\mathcal{P})_x \subset\text{At}(\mathcal{P})(D)_x$ corresponds to
\begin{gather*}
    \text{At}^{D, s_i}(\mathcal{P})_x = \left\{ \, (c,r) \in \text{Ad}(\mathcal{P})_x \oplus \mathcal{O}_{ C \times S, x}\; \mid \; c_{\kappa(x)} = r_{\kappa(x)} \, \overline{s}_i  \, \right\} \; \subset \; \text{Ad}(\mathcal{P})_x \oplus \mathcal{O}(D)_x
\end{gather*}
Diagram \ref{diagram: 5} is identified with
\[ 0 \, \longrightarrow \text{Ad}\, \mathcal{P}(-D)_x \, \xrightarrow{\; \; \; \left(\, k(-D)_x, \, 0 \, \right)\; \; \;} \, \text{At}^{D, s_i}(\mathcal{P})_x \, \xrightarrow{\; \; \; \; \text{pr}_2 \; \; \; \;} \mathcal{O}_{C \times S, x} \, \longrightarrow \, 0 \]
By the construction of $\mathcal{F}$, the given isomorphism $\Omega^1_{C \times S/ S}(D)_x \cong \mathcal{O}_{C \times S, x}$ induces an identification
\[ \mathcal{F}_x = \left\{ \, (c,r) \in \text{Ad}(\mathcal{P})_x \oplus \mathcal{O}_{C \times S, x} \; \mid \; c_{\kappa(x)} = - r_{\kappa(x)} \, \overline{s}_i  \, \right\}   \]
The submodule $\mathcal{G}_x \subset \text{At}(\mathcal{P})_x \oplus \mathcal{F}_x$ is given by
\begin{gather*}
    \mathcal{G}_x = \left\{ \, (c ,r, d, r) \in \text{Ad}(\mathcal{P})(-D)_x \oplus \mathcal{O}_{ C \times S, x}\oplus \text{Ad}(\mathcal{P})_x \oplus \mathcal{O}_{C \times S, s} \; \mid \; d_{\kappa(x)} = - r_{\kappa(x)} \, \overline{s}_i \, \right\}
\end{gather*}
The map $a_x : \mathcal{G}_x \longrightarrow \text{At}(\mathcal{P})(D)_x$ corresponds under the identification $\text{At}(\mathcal{P}(D)_x \cong \text{Ad}(\mathcal{P})_x \oplus \mathcal{O}_{C \times S, x}$ to the composition
\begin{gather*} 
a_x : \mathcal{G}_x \, \hookrightarrow \, \text{Ad}(\mathcal{P})(-D)_x \oplus \mathcal{O}_{C \times S, x} \oplus \text{Ad}(\mathcal{P})_x \oplus \mathcal{O}_{C \times S, x} \xrightarrow{\left( \, k(-D)_x \circ \text{pr}_1  -  \text{pr}_3, \, \text{pr}_2\right)}\, \text{Ad}(\mathcal{P})_x \oplus \mathcal{O}_{C \times S, x}
\end{gather*}
In other words, $a_x(c, r, d, r) = (c-d, \, r)$. This implies that the image of $a_x$ is the submodule $\text{At}^{D, s_i}(\mathcal{P})_x$. Part (C1) of the claim follows.  On the other hand the commutativity of the induced diagram
\begin{figure}[H]
\centering
\begin{tikzcd}
0  \ar[r] & \text{Ad}(\mathcal{P})(-D)_x \ar[d, symbol = \xlongequal{}] \ar[r, "{(\text{id}, \,  0, 0 ,0)}"] & \mathcal{G}_x \, / \, \text{Im} (b, \, (j, 0) \,)_x \ar[r, "p_x"] \ar[d, symbol = \xrightarrow{\sim}]  & \mathcal{O}_{x} \ar[r] \ar[d, symbol = \xlongequal{}] & 0\\
    0 \ar[r] & \text{Ad}(\mathcal{P})(-D)_x \ar[r, "{(\, k(-D)_x, \, 0 \, )}"] & \text{At}^{D, s_i}(\mathcal{P})_x \ar[r] & \mathcal{O}_{x} \ar[r] & 0
\end{tikzcd}
\end{figure}
\vspace{-0.5cm}
follows plainly from the concrete description of $a_x$ we have given. This concludes the proof of the claims.
\end{proof}
\begin{remark}
Using the notation of \cite{biswas-criterion-log-prescribed}[Prop. 3.1] when $G = \text{GL}_n$, we have $\gamma_{\mathcal{P}} = -\phi^0_{E}$, $s_i = A(x_i)$ and $\delta_i = \gamma_{x_i}$.
\end{remark}
\begin{example} \label{example: logarithmic obstruction line bundle}
Set $S = \text{Spec} \, k$ and $G = \mathbb{G}_m$. Let $\mathcal{L}$ be a line bundle on $C$. The identifications explained in Example \ref{example: atiyah class line bundle} induce a canonical isomorphisms $H^1\left(C, \, Ad \, \mathcal{L} \otimes \Omega_{C / k}^{1}\right) \cong k$ and $H^0\left(x_i, \, Ad \, \mathcal{L}|_{x_i} \right) \cong k$. Under these identifications we have that $\delta_i$ is the identity for all $i \in I$. Lemma \ref{lemma: atiyah class line bundle} and Proposition \ref{prop: atiyah class with prescribed residues} imply that $\gamma_{\mathcal{L}}^{D, s_i} = - \text{deg}\, \mathcal{L} - \sum_{i \in I} s_i$.
\end{example}
\end{subsection}
\begin{subsection}{Indecomposable example over elliptic curves}
Let $C$ be an elliptic curve over a field $k$ of characteristic not equal to $2$. Let $x \in C$ be a $k$-point, and set $D = x$. We will work with $G = SL_2$. 

Since the $k$-vector space $H^1(C, \mathcal{O}_C)$ is one-dimensional, there exists a unique nonsplit extension of vector bundles  $0 \to \mathcal{O}_{C} \to \mathcal{E} \to \mathcal{O}_{C} \to 0$. Here $\mathcal{E}$ is a vector bundle of rank $2$, which we can think of as a $GL_2$-bundle. It follows from \cite{atiyah-vector-bundles}[Lemma 16] that $\mathcal{E}$ is indecomposable as a vector bundle.
 
 There is a canonical trivialization of the determinant $  \text{det}(\mathcal{E}) \xrightarrow{\sim} \text{det}(\mathcal{O}_{C}) \otimes \text{det}(\mathcal{O}_{C}) \xrightarrow{\sim} \mathcal{O}_{C}$. This yields a reduction of structure group to $SL_2$. We denote the corresponding $SL_2$-bundle by $\mathcal{P}$.
 \begin{lemma} \label{lemma: example indecomposable}
 Suppose that the characteristic of $k$ is $0$.
\begin{enumerate}[(i)]
    \item $\mathcal{P}$ is $L$-indecomposable in the sense of \cite{bbn-krull}[Def. 2.1].
    \item Let $\text{Aut}(\mathcal{P})$ denote the automorphism group of $\mathcal{P}$ (as defined in \cite{bbn-krull}). Then all tori contained in $\text{Aut}(\mathcal{P})$ are trivial.
    \item The Atiyah class $\gamma_{\mathcal{P}}$ is $0$.
\end{enumerate}
 \end{lemma}
 \begin{proof} \quad \newline
 \noindent (i) We use \cite{bbn-krull}[Prop. 2.4]. We have to check that $\mathcal{P}$ does not admit a reduction of structure group to any proper Levi sugroup of $SL_2$. Such Levi reduction would induce a nontrivial direct sum decomposition of the vector bundle $\mathcal{E}$, contradicting the fact that $\mathcal{E}$ is indecomposable.

\noindent (ii) Since the center $Z_{G}$ of $SL_2$ is finite, \cite{bbn-krull}[Defn. 2.1] implies that any torus $T \subset \text{Aut}(\mathcal{P})$ is trivial.

\noindent (iii) It follows from \cite{bbn-krull}[Remark 3.6] that the Atiyah class $\gamma_{\mathcal{P}}$ vanishes.
 \end{proof}
 By definition, the extension bundle $0 \to \mathcal{O}_{C} \to \mathcal{E} \to \mathcal{O}_{C} \to 0$ comes from a $\mathbb{G}_a$-bundle. The Cech cocycle of $\mathcal{E}$ is obtained by choosing local splittings of the extension. This shows that the bundle $\mathcal{P}$ admits a reduction of structure group to the unipotent radical $U = \mathbb{G}_a$ of the Borel group of upper triangular matrices inside $SL_2$.
     
     Restrict the adjoint representation $\text{Ad}: SL_2 \to GL(\mathfrak{sl}_2)$ to the subgroup $U$. As a $U$-representation, $\mathfrak{sl}_2$ admits a filtration
     \[ 0 \subset \mathfrak{u} \subset \mathfrak{u} \oplus \mathfrak{t} \subset \mathfrak{sl}_2 \]
    where $\mathfrak{u}$ is the Lie algebra of $U$ and $\mathfrak{t}$ is the toral Lie algebra of diagonal matrices inside $\mathfrak{sl}_2$. Note that $U$ acts trivially on $\mathfrak{u}$, because $U$ is commutative. Therefore the associated line bundle $\mathcal{P} \times^{U} \mathfrak{u}$ is trivial. On the other hand, a matrix computation shows that the $U$-representations $\mathfrak{u} \oplus \mathfrak{t}$ and $\mathfrak{sl}_2/ \mathfrak{u}$ are isomorphic to the composition of the inclusion of $U \hookrightarrow SL_2$ and the standard representation $SL_2 \to GL_2$ (here we are using the fact that the characteristic of $k$ is not $2$). It follows that $\mathcal{P} \times^{U} (\mathfrak{u} \oplus \mathfrak{t}) \cong \mathcal{E}$ and $\mathcal{P} \times^{U} (\mathfrak{sl}_2 / \mathfrak{u}) \cong \mathcal{E}$. Therefore we have an induced filtration of $\text{Ad}(\mathcal{P})$
    \[ 0 \subset \mathcal{O}_{C} \subset \mathcal{E} \subset \text{Ad}(\mathcal{P})\]
    with $\text{Ad}(\mathcal{P})/ \mathcal{O}_{C} \cong \mathcal{E}$ and $\text{Ad}(\mathcal{P})/\mathcal{E} \cong \mathcal{O}_{C}$. The restriction of the filtration to $x$ yields full flag of the $3$-dimensional fiber $\text{Ad}(\mathcal{P})|_{x}$ by vector subspaces
    \[ 0 \subset V_1 \subset V_2 \subset \text{Ad}(\mathcal{P})|_{x}\]
     Let $\delta: H^0(C, \text{Ad}(\mathcal{P})|_{x}) \to H^1(C, \text{Ad}(\mathcal{P}) \otimes \Omega_{C}^1)$ denote the boundary morphism for the short exact sequence 
 \[ 0 \to \text{Ad}(\mathcal{P}) \otimes \Omega_{C}^1 \to \text{Ad}(\mathcal{P}) \otimes \Omega_{C}^1(x) \to \text{Ad}(\mathcal{P})|_{x} \to 0 \]
    as defined before Proposition \ref{prop: atiyah class with prescribed residues}. The following lemma gives an explicit description of $\delta$.
    \begin{lemma} \label{lemma: example elliptic curve}
    With notation as above, we have:
    \begin{enumerate}[(i)]
        \item The $k$-vector space of obstructions $H^1(C, \text{Ad}(\mathcal{P}) \otimes \Omega^1)$ has dimension $1$.
        \item The boundary morphism $\delta: H^0(C, \text{Ad}(\mathcal{P})|_{x}) \to H^1(C, \text{Ad}(\mathcal{P}) \otimes \Omega^1)$ is surjective.
        \item $V_2$ is the kernel of $\delta$.
    \end{enumerate}
    \end{lemma}
    \begin{proof}
    Consider the exact sequence in cohomology
    \begin{gather*}
    H^0(C, \text{Ad}(\mathcal{P}) \otimes \Omega^1_{C}(x)) \to H^0(C, \text{Ad}(\mathcal{P})|_{x}) \xrightarrow{\delta} H^1(C, \text{Ad}(\mathcal{P})\otimes \Omega^1_{C}) \to H^1(C, \text{Ad}(\mathcal{P}) \otimes \Omega^1_{C}(x))
    \end{gather*}
    Since $C$ is an elliptic curve, we have $\Omega^1_{C} \cong \mathcal{O}_{C}$. Using this identification plus an application of Serre duality yields the following exact sequence
    \begin{gather*}
        H^0(C, \text{Ad}(\mathcal{P})(x)) \to H^0(C, \text{Ad}(\mathcal{P})|_{x}) \xrightarrow{\delta} H^0(C, \text{Ad}(\mathcal{P})^{\vee})^* \to H^0(C, \text{Ad}(\mathcal{P})^{\vee}(-x))^*
        \end{gather*}
    Since the characteristic of $k$ is not $2$, the trace form on $\mathfrak{sl}_2$ is nondegenerate. This shows that the adjoint representation is self-dual, and hence $\text{Ad}(\mathcal{P}) \cong \text{Ad}(\mathcal{P})^{\vee}$. Therefore we can rewrite
    \begin{equation} \label{eqn: 2} H^0(C, \text{Ad}(\mathcal{P})(x)) \to H^0(C, \text{Ad}(\mathcal{P})|_{x}) \xrightarrow{\delta} H^0(C, \text{Ad}(\mathcal{P}))^* \to H^0(C, \text{Ad}(\mathcal{P})(-x))^* \end{equation}
    We now proceed with the proof of the lemma.
    
\noindent (i) The discussion above implies that $\text{dim}\, H^1(C, \text{Ad}(\mathcal{P}) \otimes \Omega_{C}^1)= \text{dim}\, H^0(C, \text{Ad}(\mathcal{P}))$. We know that we have filtration $0 \subset \mathcal{O}_{C} \subset \mathcal{E} \subset \text{Ad}(\mathcal{P})$ with $\text{Ad}(\mathcal{P})/ \mathcal{O}_{C} \cong \mathcal{E}$ and $\text{Ad}(\mathcal{P})/\mathcal{E} \cong \mathcal{O}_{C}$. This can be used to prove that
    \[  0 \to \mathcal{O}_{C} \to \text{Ad}(\mathcal{P}) \to \mathcal{E} \to 0 \] 
    is the unique indecomposable extension of $\mathcal{E}$ by $\mathcal{O}_{C}$ as in \cite{atiyah-vector-bundles}[Lemma 16]. It follows by direct computation (or \cite{atiyah-vector-bundles}[Lemma 15(i)]) that $\text{dim}\, H^0(C, \text{Ad}(\mathcal{P})) = 1$.
    \medskip
    
\noindent (ii) We show that $H^0(C, \text{Ad}(\mathcal{P})(-x)) = 0$ in the exact sequence (\ref{eqn: 2}). Equivalently, we need to prove that there are no nontrivial morphisms $\mathcal{O}_{C} \to \text{Ad}(\mathcal{P})(-x)$.
     The subquotients of the filtration $0 \subset \mathcal{O}_{C} \subset \mathcal{E} \subset \text{Ad}(\mathcal{P})$ are all isomorphic to the (stable) line bundle $\mathcal{O}_{C}$ of slope $\mu = 0$. It follows that $\text{Ad}(\mathcal{P})$ is semistable of slope $0$. Therefore $\text{Ad}(\mathcal{P})(-x)$ is a semistable vector bundle of negative slope $\mu = -1$. A standard argument now shows that there are no notrivial morphisms $\mathcal{O}_{C} \to \text{Ad}(\mathcal{P})$ from the semistable bundle $\mathcal{O}_{C}$ of slope $0$ to the semistable bundle $\text{Ad}(\mathcal{P})(-x)$ of negative slope (see e.g. \cite{huybrechts.lehn}[Lemma 1.3.3]).
     \medskip
     
\noindent (iii) By the exact sequence (\ref{eqn: 2}), we need to prove that the image of the following composition is $V_2$.
    \[H^0(C, \text{Ad}(\mathcal{P})(x)) \to H^0(C, \text{Ad}(\mathcal{P})(x)|_{x}) \xrightarrow{\sim} H^0(C, \text{Ad}(\mathcal{P})|_{x})\]
    By dimension count, it suffices to show that $\text{Im}(H^0(C, \text{Ad}(\mathcal{P})(x))) \subset V_2$. In other words, we want to show that the following composition is trivial.
    \begin{gather*} H^0(C, \text{Ad}(\mathcal{P})(x)) \to H^0(C, \text{Ad}(\mathcal{P})(x)|_{x}) \to H^0(C, (\text{Ad}(\mathcal{P})/\mathcal{E})(x)|_{x}) \xrightarrow{\sim} H^0(C, \text{Ad}(\mathcal{P})|_{x}))/V_2
    \end{gather*}
    Choose a global section $f$ in $H^0(C, \text{Ad}(\mathcal{P})(x))$. We need to prove that the image in $(\text{Ad}(\mathcal{P})/\mathcal{E})(x)|_{x}$ is $0$. By reduction, $f$ yields a global section $\overline{f}$ of the quotient $(\text{Ad}(\mathcal{P})/\mathcal{E})(x) \cong \mathcal{O}_{C}(x)$. A Riemann-Roch computation shows that $\text{dim} \, H^0(C, \mathcal{O}_{C}(x)) = \text{dim}\, H^0(C, \mathcal{O}_{C}) = 1$. This implies that the global section $\overline{f}$ must necessarily come from the subsheaf $\text{Ad}(\mathcal{P})/\mathcal{E} \cong \mathcal{O}_{C}$ of $(\text{Ad}(\mathcal{P})/\mathcal{E})(x)$. Therefore $\overline{f}$ becomes $0$ when restricted to the fiber $(\text{Ad}(\mathcal{P})/\mathcal{E})(x)|_{x}$, as desired.
    \end{proof}
    Now we use our previous discussion to characterize the residues that can arise from a logarithmic connection on $\mathcal{P}$.
    \begin{prop} \label{prop: main proposition elliptic curve example}
     Suppose that the characteristic of $k$ is $0$. The $SL_2$-bundle $\mathcal{P}$ admits a logarithmic connection with residue $s \in \text{Ad}(\mathcal{P})|_{x}$ at $x$ if and only if $s$ belongs to the vector subspace $V_2 \subset \text{Ad}(\mathcal{P})|_{x}$. \qed
    \end{prop}
    \begin{proof}
    Let $s \in H^0(C, \text{Ad}(\mathcal{P})|_{x})$. Recall from Diagram \ref{diagram: 5} that the $SL_2$-bundle $\mathcal{P}$ admits a logarithmic connection with residue $s$ if and only if the cohomological obstruction $\gamma_{\mathcal{P}}^{D, s}$ vanishes. By Proposition \ref{prop: atiyah class with prescribed residues}, we have $\gamma_{\mathcal{P}}^{D, s} = \gamma_{\mathcal{P}} - \delta(s)$. Lemma \ref{lemma: example indecomposable}(iii) shows that $\gamma_{\mathcal{P}} = 0$. Hence $\gamma_{\mathcal{P}}^{D, s} = -\delta(s)$. We now conclude by Lemma \ref{lemma: example elliptic curve}, which says that $\delta(s)= 0$ if and only if $s \in V_2$.
\end{proof}

 We consider now the setup of \cite{Biswas_2017}[\S1]. On \cite{Biswas_2017}[pg. 2] we have a maximal torus $T \subset \text{Aut}(\mathcal{P})$ and an associated Levi subgroup $H \subset G$. In our example of $SL_2$-bundle $\mathcal{P}$, Lemma \ref{lemma: example indecomposable}(i) implies that $H = SL_2$ (the identity reduction to $SL_2$ satisfies both conditions at the end of page 12 in \cite{Biswas_2017}). Furthermore, it follows from Lemma \ref{lemma: example indecomposable}(ii) that $T \subset \text{Aut}(\mathcal{P})$ is trivial. In particular every residue $s \in \text{Aut}(\mathcal{P})|_{x}$ is $T$-rigid. 
 
 Choose a residue $s \in \text{Ad}(\mathcal{P})|_{x}$ that does not lie in the $2$-dimensional vector subspace $V_2$. Then \cite{Biswas_2017}[Thm. 1.1] is false for the $SL_2$-bundle $\mathcal{P}$ and the residue $w_x = s$. Indeed, property (1) in \cite{Biswas_2017}[Thm. 1.1] is not satisfied: by Proposition \ref{prop: main proposition elliptic curve example} the bundle $\mathcal{P}$ does not admit a logarithmic connection with residue $s$ at $x$. On the other hand property (2) of \cite{Biswas_2017}[Thm. 1.1] is satisfied. This is because the only character $\chi$ of $H = SL_2$ is the trivial character, which satisfies $d\chi = 0$ and $\text{deg} \, E_{H}(\chi) = \text{deg} \, \mathcal{O}_{C} = 0$. This contradicts \cite{Biswas_2017} which says that properties (1) and (2) are equivalent. \cite{Biswas_2017}[Thm. 1.2] does not hold in this case either, because of the same reason.

The above error originates from the use of \cite{Biswas_2017}[Cor. 3.1] during the second half of the proof of \cite{Biswas_2017}[Prop. 5.1]. Let $S$ be the quotient $H/T_G$ of $H$ by its maximal central torus $T_G$, and let $E_{S}$ the $S$-bundle obtained from $E_{H}$ via the morphism $H \to S$. The reduction of structure group $E_{P} \subset E_{S}$ constructed in page 16 of \cite{Biswas_2017} need not be compatible with the residues $w_{x}^{S}$ obtained from $w_{x}$ under the induced morphism of adjoint bundles $\text{Ad}(E_{H}) \to \text{Ad}(E_{S})$. However, the application of \cite{Biswas_2017}[Cor. 3.1] requires compatibility of the reduction with the residues.

One way to repair \cite{Biswas_2017}[Thm. 1.1, 1.2] is to impose the above condition that $w_{x_i}^S \in \text{Ad}(E_{P})|_{x_i}$ for all $i \in I$. Indeed the proof of \cite{Biswas_2017}[Prop. 5.1] then goes through. As an example, consider our $SL_2$-bundle $\mathcal P$ over an elliptic curve $C$. In this case we have $S= SL_2$. The space of sections $H^0(C, \text{Ad}(\mathcal P))$ is one-dimensional, and hence there is a unique nonzero adjoint section up to scaling. This section is given by the inclusion $\mathcal{O}_{C} \subset \text{Ad}(\mathcal{P})$ in the first step of the filtration described above, so it does not vanish anywhere. The corresponding parabolic subgroup $P$ constructed in \cite{Biswas_2017}[pg. 16] is a Borel subgroup of $S$. The Borel reduction $E_P$ comes from the filtration $0 \subset \mathcal{O}_C \subset \mathcal E$, and the subspace $\text{Ad}(E_P)|_{x}$ of compatible residues coincides with the two-dimensional subspace $V_2 \subset \text{Ad}(\mathcal P)|_{x}$. This is consistent with our Proposition \ref{prop: main proposition elliptic curve example}.

Another possible way to repair \cite{Biswas_2017}[Thm. 1.1, 1.2] would be to impose the additional requirement that appears in \cite{Biswas_2017}[pg. 15] that $\beta = 0$. In our notation, this translates to the condition that $\gamma_{E_{S}}^{D, w^{S}_{x_i}}= -\sum_{i \in I} \delta_i(w^{S}_{x_i})$ is $0$ in $H^1(C, \text{Ad}(E_{S}) \otimes \Omega_{C}^1)$. Note that the bundle $E_S$ is $L$-indecomposable by \cite{bbn-krull}[Thm. 3.2], and hence $\gamma_{E_S}=0$.
\end{subsection}
\begin{subsection}{Functoriality for logarithmic $G$-connections} \label{subsection: functoriality logarithmic}
Logarithmic connections enjoy the same functoriality properties as regular connections. Let $f: T \longrightarrow S$ be a morphism of $k$-schemes. Let $\mathcal{P}$ be a $G$-bundle on $C \times S$. The same reasoning as for regular connections shows that the logarithmic Atiyah class $\gamma^D_{\mathcal{P}}$ behaves well under pullback. Explicilty, $\gamma^D_{(Id_C \times f)^{*} \mathcal{P}} = (Id_C \times f)^{*} \gamma^D_{\mathcal{P}}$. Here we are abusing notation and writing $(Id_C \times f)^{*}$ on the right-hand side to denote the natural map on the corresponding sheaf cohomology groups. It follows that the logarithmic Atiyah sequence for $(Id_C \times f)^{*} \mathcal{P}$ is the $(Id_C \times f)$-pullback of the Atiyah sequence for $\mathcal{P}$. We can therefore pullback logarithmic connections by interpreting them as splittings of the logarithmic Atiyah sequence. Let $\theta$ be a logarithmic $G$-connection on $\mathcal{P}$ with poles on $D$. We denote by $f^{*}\theta$ the corresponding logarithmic connection on $(Id_C \times f)^{*} \mathcal{P}$. This construction is compatible with taking residues. This implies that the obstruction classes $\gamma^{D, s_i}_{\mathcal{P}}$ and the corresponding Atiyah sequences with prescribed residues are compatible with pullback. We leave the precise formulation of these statements to the interested reader.
    
Logarithmic connections have covariant functoriality with respect to morphisms of structure groups. Fix a $k$-scheme $S$. Let $H$ be a linear algebraic group over $k$ with Lie algebra $\mathfrak{h}$. Let $\varphi: G \longrightarrow H$ be a homomorphism of algebraic groups. Let $\mathcal{P}$ be a $G$-bundle over $C \times S$. Fix a set of sections $s_i \in H^0\left(x_i \times S, \, Ad \, \mathcal{P}|_{x_i \times S} \right)$. Recall that there is a naturally defined morphism of vector bundles $Ad \, \varphi : Ad \, \mathcal{P} \longrightarrow Ad \, \varphi_{*} \mathcal{P}$. Set $\varphi_{*} s_i \vcentcolon = Ad \, \varphi|_{x_i \times S} \,(s_i) \in H^0\left( x_i \times S, \, Ad \, \mathcal{P} |_{x_i \times S}\right)$. The description of $\gamma_{\mathcal{P}}^{D, s_i}$ in Proposition \ref{prop: atiyah class with prescribed residues} and the discussion in Subsection \ref{subsection: functoriality regular connections} imply that $\gamma_{\varphi_{*}\mathcal{P}}^{D, \, \varphi_{*}s_i} = \varphi_{*}^1( \gamma_{\mathcal{P}}^{D, s_i})$. Here recall that $\varphi^1_{*}: H^1\left( C, \, Ad\, \mathcal{P} \otimes\Omega^1_{C\times S/ S}\right) \longrightarrow H^1\left( C, \, Ad\, \varphi_{*} \mathcal{P} \otimes\Omega^1_{C\times S/ S}\right)$ is the map on cohomology groups induced by $Ad \, \varphi$. The concrete description of $\varphi_{*}^1( \gamma_{\mathcal{P}}^{D, s_i})$ in terms of pushouts \cite[\href{https://stacks.math.columbia.edu/tag/010I}{Tag 010I}]{stacks-project} shows that there is a commutative diagram of Atiyah sequences with prescribed residues
\begin{figure}[H]
\centering
\begin{tikzcd}
    0  \ar[r]& \text{Ad} \, \mathcal{P} \otimes \Omega^1_{C \times S/S} \ar[d, "Ad \, \varphi \, \otimes \, id"] \ar[r] & At^{D, s_i}(\mathcal{P}) \ar[r] \ar[d, "At^{D,\, s_i}(\varphi)"]  & \mathcal{O}_{C \times S} \ar[r] \ar[d, symbol =  \xlongequal{}] & 0\\
    0 \ar[r] & \text{Ad} \, \varphi_{*} \mathcal{P} \otimes \Omega^1_{C \times S/S} \ar[r] & At^{D, \, \varphi_{*} s_i}(\varphi_{*}\mathcal{P}) \ar[r] & \mathcal{O}_{C \times S} \ar[r] & 0
\end{tikzcd}
\end{figure}
\vspace{-0.5cm}
We denote by $\text{At}^{D, \, s_i}(\varphi)$ the middle vertical arrow in the diagram above. Let $\theta$ be a logarithmic $G$-connection on $\mathcal{P}$ with residue $s_i$ at each $x_i$, viewed as a splitting of the top row. We compose it with $At^{D, \, s_i}(\varphi)$ in order to define a logarithmic $H$-connection $\varphi_{*} \theta \vcentcolon = At^{D, \, s_i}(\varphi) \circ \theta$ on the $H$-bundle $\varphi_{*} \mathcal{P}$ with residue $\varphi_{*} s_i$ at $x_i$.
\begin{example} \label{example: atiyah class degree for gln}
Let $S = \text{Spec} \, k$ and $G = \text{GL}_n$. Consider the determinant character $\text{det}: \text{GL}_n \longrightarrow \mathbb{G}_m$. The corresponding Lie algebra map $\text{Lie} (\text{det}) : \mathfrak{gl}_n \longrightarrow \mathfrak{gl}_1$ is given by taking the trace of a matrix. Let $\mathcal{E}$ be a vector bundle of rank $n$ on $C$, viewed as a $\text{GL}_n$-bundle. Fix an endomorphisms $s_i \in \text{End}(\mathcal{E}|_{x_i})$ for each $i \in I$. The corresponding element $\text{det}_{*} s_i \in \mathfrak{gl}_1$ is given by the trace $\text{tr} \, s_i$. Supose that $\mathcal{E}$ admits a logarithmic connection $\theta$ with residue $s_i$ at $x_i$. By functoriality, there is an induced connection $\text{det}_{*} \, \theta$ on $\text{det} \,\mathcal{E}$ with residue $\text{tr} \, s_i$ at $x_i$. Therefore the Atiyah class with prescribed residues $\gamma_{\text{det} \, \mathcal{E}}^{D, \, \text{tr} \,s_i}$ must vanish. It follows from Example \ref{example: logarithmic obstruction line bundle} that $\text{deg} \, \mathcal{E} = -\sum_{i \in I} \text{tr} \, s_i$.
\end{example}
\end{subsection}
\begin{subsection}{Stacks of logarithmic $G$-connections}
We describe the stack of logarithmic $G$-connections with a fixed set of poles. Let $D$ be a reduced divisor $D = \sum_{i \in I} x_i$, where $x_i \in C(k)$.
\begin{defn}
The moduli stack $\text{Conn}_G^D(C)$ is the pseudofunctor from $k$-schemes $S$ to groupoids given as follows. For every $k$-scheme $S$, we define
\[ \text{Conn}_G^D(C) \, (S) \vcentcolon =  \; \left\{ \begin{matrix} \text{groupoid of $G$-torsors} \; \mathcal{P} \text{ over } C \times S \\ $+$\\ \text{ a logarithmic $G$-connection $\theta$ on $\mathcal{P}$ with poles at $D$ } \end{matrix} \right\} \]
\end{defn}
We will want to keep track of the residues of the logarithmic connection. In order to do this we define another stack over $\text{Bun}_G(C)$.
	 
For every $G$-bundle $\mathcal{P}$ on a $k$-scheme $S$, we can form the adjoint bundle of $Ad\, \mathcal{P}$. This construction allows us to define a vector bundle $\mathcal{V}$ over $\text{B}G$. The pullback of $\mathcal{V}$ under a map $\mathcal{P}: S \rightarrow \text{B}G$ is the associated bundle $Ad \, \mathcal{P}$ on $S$. The total space of the vector bundle is represented by the map of stacks 
\[\pi : \left[ G \, \backslash \, \mathfrak{g} \right] \longrightarrow  \left[ G \, \backslash \, \text{Spec} \, k \right] = \text{B}G\]
Here the action in the left-most quotient stack is given by the adjoint representation.
	
For each $i \in I$, there is a morphism of algebraic stacks $\psi_i: \text{Bun}_{G}(C) \longrightarrow \text{B}G$ defined by taking the fiber at $x_i$. More precisely, for every $k$-scheme $S$ and every $G$-bundle $\mathcal{P}$ over $C \times S$, set $\psi_i (\mathcal{P}) = \mathcal{P}|_{x_i \times S}$.
\begin{defn}
 The stack $\text{Bun}_{G}^{Ad, \, D}(C)$ is defined to be the fiber product
 \[\xymatrix{
\text{Bun}_{G}^{Ad, \, D}(C) \ar[r] \ar[d] & \text{Bun}_{G}(C) \ar[d]^{\prod_{i \in I} \psi_i}\\
\prod_{i \in I} \left[ G \, \backslash \, \mathfrak{g} \right] \ar[r]^{\; \; \; \; \prod_{i \in I} \pi \; \;} & \prod_{i \in I} \text{B}G } \]
\end{defn}
\vspace{-0.25cm}
The definition implies that $\text{Bun}_{G}^{Ad, \, D}(C)$ is an algebraic stack that is locally of finite type over $k$. For each $k$-scheme $S$, the groupoid $\text{Bun}_{G}^{Ad, \, D}(C) \,  (S)$ is naturally equivalent to the groupoid of $G$-bundles $\mathcal{P}$ on $C \times S$ along with a global section of the associated bundle $Ad \left(\mathcal{P}|_{x_i \times S}\right)$ for each $i \in I$. There is a natural map $Forget^D: \text{Conn}^D_G(C) \longrightarrow \text{Bun}_{G}^{Ad, \, D}(C)$ given by forgetting the connection and remembering only the $G$-bundle and the residue at each point $x_i$. We have the following result, analogous to Proposition \ref{prop: moduli of connections affine bundle}.
\begin{prop} \label{prop: moduli logarithmic connections affine bundle}
The map $Forget^D: \text{Conn}^D_G(C) \longrightarrow \text{Bun}_{G}^{Ad, \, D}(C)$ is schematic, affine and of finite type. In particular, $\text{Conn}^D_G(C)$ is an algebraic stack that is locally of finite type over $k$.
\end{prop}
\begin{proof}
Let $S$ be a $k$-scheme. Let $(\mathcal{P}, s_i) : \, S \longrightarrow \text{Bun}_{G}^{Ad, \, D}(C)$ be a morphism. It consists of the data of a $G$-bundle $\mathcal{P}$ on $C \times S$ and a section $s_i \in H^0\left( x_i \times S, \, Ad(\mathcal{P}|_{x_i \times S}) \, \right)$ for all $i \in I$. We want to show that the projection $\text{Conn}^D_{G}(C) \times_{\text{Bun}_{G}^{Ad, \, D}(C)} \, S \, \longrightarrow \, S$ is schematic, affine and of finite type. There is a short exact sequence as in Diagram \ref{diagram: 5} above.
 \vspace{-0.25cm}
\begin{figure}[H]
\centering
\begin{tikzcd}
    0  \ar[r] & \text{Ad} \, \mathcal{P} \otimes \Omega^1_{C \times S/S}  \ar[r] & At^{D, s_i}(\mathcal{P}) \ar[r] & \mathcal{O}_{C \times S} \ar[r] & 0
\end{tikzcd}
\end{figure}
\vspace{-0.5cm}
This corresponds to a torsor for the vector bundle group scheme $\mathcal{P} \otimes \Omega^1_{C \times S/S}$. We denote by $X \to C \times S$ the total space of this torsor. Note that the morphism $X \to C\times S$ is schematic, affine and of finite presentation. This follows from \cite[\href{https://stacks.math.columbia.edu/tag/0245}{Tag 0245}]{stacks-project}+ \cite[\href{https://stacks.math.columbia.edu/tag/02L5}{Tag 02L5}]{stacks-project}+ \cite[\href{https://stacks.math.columbia.edu/tag/02L0}{Tag 02L0}]{stacks-project}, since \'etale locally on $C \times S$ we have that $X$ is isomorphic to the total space of the pullback of the vector bundle $Ad \, \mathcal{P} \otimes \Omega^1_{C \times S / S}$, which is relatively affine and of finite presentation. By the same reasoning as in \Cref{prop: moduli of connections affine bundle}, we have an identification $\text{Conn}^D_{G}(C) \times_{\text{Bun}_{G}^{Ad, \, D}(C)} \, S \cong \Gamma_{C\times S/S}(X)$ of functors over $S$. \Cref{lemma: representability of sections} implies that $\Gamma_{C\times S/S}(X) \to S$ is affine and of finite type, as desired.
\end{proof}
It is not necessarily true that the moduli stack of logarithmic connections $\text{Conn}_G^D(C)$ is quasicompact, even when $\text{char} \, k = 0$. We illustrate this with an example of an unbounded family on $\mathbb{P}^1$.
\begin{example} \label{example: unbounded family}
Let $C = \mathbb{P}^1_k = \text{Proj}(k[x_0, x_1])$. Set $G = \mathbb{G}_m$. We choose a single puncture $\infty$ in $\mathbb{P}^1_k$. For each $n \geq 1$, we equip the line bundle $\mathcal{O}(-n)$ with a logarithmic connection as follows. There is a natural inclusion
\[  \mathcal{O}(-n) \xrightarrow{x_1^n} \mathcal{O}_{\mathbb{P}^1_k}\]
that identifies $\mathcal{O}(-n)$ with the ideal sheaf of the divisor $n \infty$. We compose this with the universal derivation $d: \mathcal{O}_{\mathbb{P}^1_{k}} \rightarrow \Omega_{\mathbb{P}^1_k/k}^1$ and the natural inclusion  $\Omega_{\mathbb{P}^1_k/k}^1 \hookrightarrow \Omega_{\mathbb{P}^1_k/k}^1(\infty)$ in order to obtain a morphism
\[\partial_n :\mathcal{O}(-n) \hookrightarrow  \mathcal{O}_{\mathbb{P}^1_k/k} \xrightarrow{d} \Omega_{\mathbb{P}^1_k/k}^1 \hookrightarrow \Omega_{\mathbb{P}^1/k}^1(\infty)\]
This can be described in coordinates in the following way. Over the affine open $\mathbb{A}^1_{\frac{x_0}{x_1}}$ we have
\[ \partial_n\left( f \cdot \frac{1}{x_1^n}\right) = df \]
On the other hand, over the affine open $\mathbb{A}^1_{\frac{x_1}{x_0}}$ we have
 \[ \partial_n\left( f \cdot \frac{1}{x_0^n}\right) = \left(\frac{x_1}{x_0}\right)^n \otimes df + n f \left(\frac{x_1}{x_0}\right)^n \otimes \frac{d\left(\frac{x_1}{x_0}\right)}{\left(\frac{x_1}{x_0}\right)}  \]
This description shows that $\partial_n$ factors through the subsheaf $\mathcal{O}(-n) \otimes \Omega_{\mathbb{P}_1/k}^1(\infty) \subset \Omega_{\mathbb{P}^1/k}^1(\infty)$. This way we obtain a logarithmic connection 
\[\partial_n : \mathcal{O}(-n) \rightarrow \mathcal{O}(-n) \otimes \Omega_{\mathbb{P}^1/k}^1(\infty)\] defined on $\mathcal{O}(-n)$. The residue at the pole is $n$.
\end{example}

This example generalizes to the case when $G$ is an arbitrary nontrivial connected reductive group.
\begin{lemma} \label{lemma: stack of logarithmic connections is not quasicompact}
Let $G$ be a nontrivial connected reductive group over $k$. Suppose that $D$ is nonempty. The stack of logarithmic connections $\text{Conn}_{G}^D(C)$ is not of finite type over $k$.
\end{lemma}
\begin{proof}
Since the property of being of finite type is stable under base-change, we can replace the ground field $k$ with an algebraic closure. For the rest of the proof we assume that $k$ is algebraically closed. 

Let $T$ be a maximal torus in $G$. By assumption, $T$ is a nontrivial split torus. We fix an identification $T \xrightarrow{\sim} \mathbb{G}_m^n$ for some $m\geq 1$. A $T$-bundle on $C$ is equivalent to a $n$-tuple $\vec{\mathcal{L}} = (\mathcal{L}_j)_{j=1}^n$ of line bundles on $C$. For such $n$-tuple $\vec{\mathcal{L}}$, we define a $n$-tuple of integers $\vec{\text{deg}}(\vec{\mathcal{L}})\vcentcolon = (\text{deg} \, \mathcal{L}_j)_{j=1}^n$. The tuple $\vec{\text{deg}}(\vec{\mathcal{L}})$ can be interpreted more canonically as a cocharacter in $X_{*}(T)$. This is the degree of the $T$-bundle $\vec{\mathcal{L}}$ as in \cite{schieder-hnstratification}[2.2.2].

Choose a Borel subgroup of $G$ containing $T$. Let $\lambda \in X_{*}(T)$ be a cocharacter of $T$ contained in the interior of the cone of $B$-dominant cocharacters. Using the identification $T \xrightarrow{\sim} \mathbb{G}_m^n$ we think of $\lambda$ as a tuple of integers $(m_j)_{j=1}^n$. Since $k$ is algebraically closed, there exist line bundles $\mathcal{L}_j^{\lambda}$ such that $\text{deg} \, \mathcal{L}_j^{\lambda} = m_j$ for each $j$. We can think of the tuple $\vec{\mathcal{L}}^{\lambda} = (\mathcal{L}_j)_{j=1}^n$ as a $T$-bundle.

Let $i \in I$. By Example \ref{example: logarithmic obstruction line bundle}, there is a canonical identification $H^0(x_i, \text{Ad}(\mathcal{L}_j^{\lambda}|_{x_i}) \cong k$. Set $s_{i,j} = -\frac{1}{|I|}\text{deg}\, \mathcal{L}^{\lambda}_{j}$. Example \ref{example: logarithmic obstruction line bundle} shows that the obstruction $\gamma_{\mathcal{L}_j^{\lambda}}^{D, s_{i,j}}$ satisfies $\gamma_{\mathcal{L}_j^{\lambda}}^{D, s_{i,j}} = 0$. Therefore the line bundle $\mathcal{L}_j^{\lambda}$ admits a logarithmic connection $\theta_{j}^{\lambda}$ with residue $s_{i,j}$ at $x_i$. The pair of $n$-tuples 
\[\left(\vec{\mathcal{L}}^{\lambda}, \vec{\theta}^{\lambda} \right)\vcentcolon =\left(\mathcal{L}_j^{\lambda}, \theta_j^{\lambda} \right)_{j=1}^n\]
is a $T$-bundle equipped with a logarithmic connection with poles at $D$.

Let $\rho$ denote the inclusion $\rho: T \hookrightarrow G$. By functoriality of logarithmic connections, there is an associated $G$-bundle with logarithmic connection \[(\mathcal{G}^{\lambda}, \nu^{\lambda}) \vcentcolon = \rho_*\left( \,\left(\vec{\mathcal{L}}^{\lambda}, \vec{\theta}^{\lambda}\right) \, \right)\]
This way we get an infinite set of $G$-bundles $\{\mathcal{G}^{\lambda}\}_{\lambda}$ indexed by the cocharacters $\lambda \in X_{*}(T)$ that are contained in the interior of the $B$-dominant cone. Each $\mathcal{G}^{\lambda}$ admits a logarithmic connection $\nu^{\lambda}$ with poles at $D$. In order to conclude the proposition, we shall show that the set $\{\mathcal{G}^{\lambda}\}_{\lambda}$ is unbounded in $\text{Bun}_{G}(C)$.

For each $\lambda$, we have $\mathcal{G}^{\lambda} = \rho_*(\vec{\mathcal{L}}^{\lambda})$. Let $\psi$ denote the inclusion $\psi: T \hookrightarrow B$. The $T$-bundle $\vec{\mathcal{L}}^{\lambda}$ is semistable in the sense of \cite{schieder-hnstratification}[2.2.3], because $T$ does not have nontrivial proper parabolic subgroups. By construction, $\lambda$ is the degree of $\psi_*(\vec{\mathcal{L}}^{\lambda})$ as in \cite{schieder-hnstratification}[2.2.2]. Since $\lambda$ is in the interior of the $B$-dominant cone, it follows that $\psi_*(\vec{\mathcal{L}}^{\lambda})$ is the canonical reduction of $\mathcal{G}^{\lambda}$ \cite{schieder-hnstratification}[2.4.1]. Therefore the cocharacter $\lambda$ is the Harder-Narasimhan type of $\mathcal{G}^{\lambda}$. Hence the $G$-bundles $\{\mathcal{G}^{\lambda}\}_{\lambda}$ have distinct Harder-Narasimhan types. Since the image of each Harder-Narasimhan stratum is constructible in $\text{Bun}_{G}(C)$ \cite{schieder-hnstratification}[Thm. 2.3.3(a)], it follows that the set $\{\mathcal{G}^{\lambda}\}_{\lambda}$ must be unbounded.
\end{proof}

Let's return to the case when $G$ is an arbitrary smooth connected linear algebraic group. Using Lemma \ref{lemma: stack of logarithmic connections is not quasicompact} we can prove the following.
\begin{prop} \label{prop: characterization of quasicompactness}
Suppose that $D$ is nonempty and that the characteristic of $k$ is $0$. Then the stack $\text{Conn}_{G}^{D}(C)$ is of finite type over $k$ if and only if the group $G$ is unipotent.
\end{prop}
\begin{proof}
Assume that the group $G$ is unipotent. Consider the composition
\[ h_{G} : \text{Conn}^{D}_{G}(C) \xrightarrow{\text{Forget}^{D}} \text{Bun}_{G}^{Ad, D}(C) \rightarrow \text{Bun}_{G}(C)\]
By definition, $\text{Bun}_{G}^{Ad, D}(C) \rightarrow \text{Bun}_{G}(C)$ is of finite type. Proposition \ref{prop: moduli logarithmic connections affine bundle} shows that $\text{Forget}^{D}$ is also of finite type. Therefore, the composition $h_{G}$ is of finite type. By Proposition \ref{prop: quasicompactness associated bundle morphism unipotent} (iv), $\text{Bun}_{G}(C)$ is of finite type over $k$. We conclude that $\text{Conn}_{G}^{D}(C)$ is of finite type over $k$.

Conversely, suppose that $G$ is not unipotent. Let $U$ denote the unipotent radical of $G$. We have a short exact sequence of algebraic groups
\[ 1 \rightarrow U \rightarrow G \rightarrow \overline{G} \rightarrow 1\]
where $\overline{G}$ is a nontrivial connected reductive group. Since the characteristic of $k$ is $0$, the short exact sequence above admits a splitting. Therefore we can view $G = U \rtimes \overline{G}$. Consider the chain of morphisms
\[ \overline{G} \xhookrightarrow{ \iota} U \rtimes \overline{G} = G \xtwoheadrightarrow{q} \overline{G} \]
By functoriality of logarithmic connections, this induces a chain of morphisms of stacks
\[ \text{Conn}^{D}_{\overline{G}}(C) \xrightarrow{\iota_{*}} \text{Conn}_{G}^{D}(C) \xrightarrow{q_{*}} \text{Conn}_{\overline{G}}^{D}(C) \]
By definition, the composition $q_{*} \circ \iota_{*}$ is the identity. Hence $\iota_{*}$ exhibits $\text{Conn}^{D}_{\overline{G}}(C)$ as a subfunctor of $\text{Conn}^{D}_{G}(C)$. Assume for the sake of contradiction that $\text{Conn}^{D}_{G}(C)$ is quasicompact. This would imply that the substack $\text{Conn}^{D}_{G}(C)$ is quasicompact, thus contradicting Lemma \ref{lemma: stack of logarithmic connections is not quasicompact}.
\end{proof}
For each $i \in I$, choose an orbit $O_{i}$ for the adjoint action of $G$ on its Lie algebra $\mathfrak{g}$. Each orbit $O_i$ is a smooth locally closed subscheme of $\mathfrak{g}$. After quotienting by the action of $G$, we get a locally closed substack
\[\prod_{i \in I} \left[ G \, \backslash \, O_i \right] \, \hookrightarrow \, \prod_{i \in I}\left[ G \, \backslash \, \mathfrak{g} \right] \]
\begin{defn}
$\text{Conn}_G^{D, O_i}(C)$ is defined to be the locally closed substack of $\text{Conn}_G^{D}(C)$ given by the following cartesian diagram
 \[\xymatrix{
	    \text{Conn}_G^{D, O_i}(C) \ar[r] \ar[d] & \text{Conn}^D_{G}(C) \ar[d]^{Forget^D}\\
		\prod_{i \in I} \left[ G \, \backslash \, O_i \right] \times_{\prod_{i \in I}\left[ G \, \backslash \, \mathfrak{g} \right]} \text{Bun}_{G}^{\text{Ad}, D}(C) \ar[r] & \text{Bun}_{G}^{Ad, \, D}(C)} \]
\end{defn}
The stack $\text{Conn}_G^{D, O_i}(C)$ parametrizes $G$-bundles with logarithmic connections whose residue at $x_i$ lies in the orbit $O_i$.
 \begin{thm} \label{thm: quasicompactness logarithmic connections}
Suppose that $\text{char} \, k= 0$. The algebraic stack $\text{Conn}_G^{D, O_i}(C)$ is of finite type over $k$.
 \end{thm}
 \begin{proof}
 By Proposition \ref{prop: moduli logarithmic connections affine bundle}, the map $Forget^D$ is of finite type. By definition the map $\text{Bun}^{Ad, \, D}_G(C) \longrightarrow \text{Bun}_G(C)$ is of finite type. Hence it suffices to show that the composition
 \[ h_G: \text{Conn}_G^{D, O_i}(C) \hookrightarrow \text{Conn}_G^{D}(C) \xrightarrow{Forget^D} \text{Bun}_G^{Ad, \, D}(C) \longrightarrow \text{Bun}_G(C) \]
  factors through a quasicompact open substack of $\text{Bun}_{G}(C)$.
 
 Let $U$ be the unipotent radical of $G$. We denote by $\rho_{G}: G \rightarrow G/U$ the corresponding quotient morphism. Choose a faithful representation $\rho: G/U \longrightarrow \text{GL}_n$ of the reductive group $G/U$. Let $(\rho \circ \rho_{G})_{*}O_i$ denote the unique $\text{GL}_n$-orbit in $\mathfrak{gl}_n$ containing the image $(\rho \circ \rho_{G})(O_i)$. The functoriality properties described in Subsection \ref{subsection: functoriality logarithmic} imply that there is a commutative diagram
\[\xymatrix{
		\text{Conn}_{G}^{D, O_i}(C) \ar[r] \ar[d]^{h_{G}} & \text{Conn}_{G/U}^{D, \, (\rho_{G})_*(O_i)}(C) \ar[r] \ar[d]^{h_{G/U}} & \text{Conn}_{\text{GL}_n}^{D, \, (\rho \circ \rho_{G})_{*}(O_i)}(C) \ar[d]^{h_{\text{GL}_n}}\\ \text{Bun}_{G}(C) \ar[r]^{(\rho_{G})_*} &
		\text{Bun}_{G/U}(C) \ar[r]^{\rho_{*}} & \text{Bun}_{\text{GL}_n}(C) } \]
The bottom horizontal maps are of finite type by Propositions \ref{prop: quasicompactness associated bundle map} and \ref{prop: quasicompactness associated bundle morphism unipotent}. Therefore it suffices to show that $\text{Conn}^{D, \, (\rho \circ \rho_{G})_{*} O_i}_{\text{GL}_n}(C)$ is quasicompact. We shall establish that $h_{\text{GL}_n}$ factors through a quasicompact open substack of $\text{Bun}_{\text{GL}_n}(C)$. Just as in the proof of Proposition \ref{prop: quasicompactness regular connections}, it suffices to check that the set of Harder-Narasimhan polygons of geometric points $\overline{t} \longrightarrow \text{Bun}_{\text{GL}_n}(C)$ with nonempty fiber $\text{Conn}^{D, O_i}_{\text{GL}_n}(C)_{\overline{t}}$ is a finite set.
		
Let $\mu = (\mu_1 \geq \mu_2\geq ...\geq \mu_n)$ be a tuple of rational numbers in $\frac{1}{n!}\mathbb{Z}$. Let $\mathcal{E}: \overline{t} \longrightarrow \text{Bun}_{\text{GL}_n}(C)$ be a vector bundle on $C\times \overline{t}$, where $\overline{t}$ is $\text{Spec}$ of an algebrically closed field. Suppose that the Harder-Narasimhan polygon of $\mathcal{E}$ is given by $\mu$. Assume that the fiber $\text{Conn}^{D, O_i}_{\text{GL}_n}(C)_{\overline{t}}$ is nonempty. This means that $\mathcal{E}$ admits a logarithmic connection $\theta$ with residue $\text{Res}_{x_i}  \theta$ in the conjugacy class $O_i|_{\kappa(\overline{t})}$. Set $s_i \vcentcolon = \text{Res}_{x_i}  \theta$.

After choosing a trivialization of $\mathcal{E}|_{x_i \times \overline{t}}$, the endomorphism $s_i$ is represented by a matrix in $\mathfrak{gl}_n \otimes \kappa(\overline{t})$. Let $(\lambda_i^{(l)})_{l = 1}^n$ be the tuple of eigenvalues of $s_i$. This tuple does not depend up to permutation on the choice of trivialization of $\mathcal{E}|_{x_i \times \overline{t}}$.

Let $B(\{1, 2, ..., n\})$ denote the power set of $\{1, 2, ..., n\}$. Define the set of $I$-tuples
\[ A \vcentcolon = \left\{ (J_i)_{i \in I} \in B(\{1, 2, ... , n\})^{I} \, \mid \, J_i \neq \emptyset \; \text{for all $ i \in I$} \; \; \, \text{and} \; \,  \sum_{ i \in I} \sum_{j \in J_i} \lambda_{i}^{(j)} \in \mathbb{Z} \right\} \]
Set $M \vcentcolon = \underset{ (J_i) \in A}{\text{max}} \, \left\lvert \sum_{i \in I}  \sum_{j \in J_i} \lambda_i^{(j)}\right\rvert$. The quantity $M$ does not depend on the choice of $\mathcal{E}$, because the eigenvalues are completely determined by the conjugacy classes $(\rho \circ \rho_{G})_{*}(O_i)$ of matrices. We claim that $\mu$ satisfies the following two conditions
\begin{enumerate}[(a)]
    \item $\sum_{j =1}^n \mu_j = - \sum_{i \in I}\text{tr} \, s_i$.
    \item $\mu_{1} \leq \left(\text{max}\{4g - 3 -|I|, 0\}\right) \cdot n + M$.
\end{enumerate}
The claim implies that there are finitely many possibilities for the Harder-Narasimhan polygon $\mu \in \frac{1}{n!}\mathbb{Z}^n$. Indeed, (a) and (b) imply the following chain of inequalities
\[ - \sum_{i \in I} \text{tr}\, s_i = \mu_n + \sum_{i=1}^{n-1} \mu_i  \leq \mu_n + (n-1) \cdot \mu_1 \leq \mu_n + (n-1) \cdot\left(\left(\text{max}\{4g - 3 -|I|, 0\}\right) \cdot n + M\right)\]
This yields a lower bound $\mu_n \geq -(n-1) \cdot\left(\left(\text{max}\{4g - 3 -|I|, 0\}\right) \cdot n + M\right)- \sum_{i \in I} \text{tr}\, s_i$. It would follow that the rational numbers
\begin{gather*}
-(n-1) \cdot\left(\left(\text{max}\{4g - 3 -|I|, 0\}\right) \cdot n + M\right)- \sum_{i \in I} \text{tr}\, s_i \leq \mu_n \leq \mu_{n-1} \leq \ldots \leq \mu_1 \leq  \left(\text{max}\{4g - 3 -|I|, 0\}\right) \cdot n + M
\end{gather*}
are restricted to finitely many possibilities, since they must lie in the lattice $\frac{1}{n!} \mathbb{Z}$. Hence it is sufficient to show the claims (a) and (b).

\noindent \textit{Proof of (a):} This follows from the discussion in Example \ref{example: atiyah class degree for gln}, because $\sum_{j =1}^n \mu_j = \text{deg} \, \mathcal{E}$.
\medskip

\noindent \textit{ Proof of (b):} Set $m \vcentcolon = \text{max}\{4g-3-|I|, 0\}$. Suppose that the Harder-Narasimhan filtration of $\mathcal{E}$ is given by
\[0 \subset \mathcal{F}_1 \subset \mathcal{F}_2 \subset ... \subset \mathcal{F}_l = \mathcal{E}   \]
 We use the interpretation of logarithmic $\text{GL}_n$-connections as $\Lambda$-modules, as in \cite{Simpson-repnI}[\S 2]. Let $\mathcal{D}_{C\times \overline{t}/ \, \overline{t}}$ denote the usual sheaf of rings of differential operators on $C\times \overline{t}$ \cite{Simpson-repnI}[\S 2, pg. 85]. Define $\Lambda$ to be the subsheaf of rings of $\mathcal{D}_{C\times \overline{t}/ \, \overline{t}}$ generated by $\Omega^{1}_{C\times \overline{t} / \overline{t}}(D)^{\vee} \subset T_{C\times \overline{t} / \overline{t}}$ (see \cite{Simpson-repnI}[\S2, pg. 87]). Recall that there is a filtration of $\mathcal{D}_{C\times \overline{t}/ \, \overline{t}}$ given by the order of the differential operator. This induces a filtration on the subsheaf $\Lambda \subset \mathcal{D}_{C\times \overline{t}/ \, \overline{t}}$, which endows $\Lambda$ with the structure of a sheaf ring of differential operators as defined in \cite{Simpson-repnI}[\S 2, pg.77]. The first graded piece $\text{Gr}_1 \Lambda$ associated to the filtration is $\text{Gr}_1 \Lambda = \Omega_{C\times \overline{t} / \, \overline{t}}^{1}(D)^{\vee}$. Observe that the twist $\Omega_{C\times \overline{t} / \, \overline{t}}^{1}(D)^{\vee}(m)$ is generated by global sections. This follows from a standard cohomology argument using Serre duality. The proof of Lemma 3.3 in \cite{Simpson-repnI}[\S3] shows that $\mu_1 = \mu(\mathcal{F}_1) \leq mn + \mu(\mathcal{G})$, where $\mathcal{G}$ is the smallest subbundle containing $\mathcal{F}_1$ and preserved by the logarithmic connection $\theta$. We will show that any subsheaf $\mathcal{G} \subset \mathcal{E}$ preserved by the logarithmic connection satisfies $\mu(\mathcal{G}) \leq M $.

Restrict the logarithmic connection $\theta$ to the subbundle $\mathcal{G}$ in order to obtain a logarithmic connection $\theta^{\mathcal{G}}$ on $\mathcal{G}$. By assumption the residue $s_i \in \text{End}(\mathcal{E}|_{x_i \times \overline{t}})$ preserves the subspace $\mathcal{G}|_{x_i \times \overline{t}}$. The residue $s_i^{\mathcal{G}} \vcentcolon = \text{Res}_{x_i} \, \theta^{\mathcal{G}}$ is the restriction of $s_i$ to $\mathcal{G}|_{x_i \times \overline{t}}$. For all $i \in I$, there is a nonempty subset $J_i \subset \{1, 2, ... , n\}$ such that the eigenvalues of $s_i^{\mathcal{G}}$ are $(\lambda_i^{(j)})_{j \in J_i}$.

The discussion in Example \ref{example: atiyah class degree for gln} applied to $\mathcal{G}$ shows that
\[ \text{deg} \, \mathcal{G} = -\sum_{i \in I} \text{tr} \, s_i^{\mathcal{G}} = -\sum_{i \in I} \sum_{j \in J_i} \lambda_i^{(j)}\]
Since $\text{deg} \, \mathcal{G} \in \mathbb{Z}$, we have $(J_i)_{i \in I} \in A$. By definition, the right hand side is bounded by $M$. This implies that $\mu(\mathcal{G}) \leq M$, thus concluding the proof of the claim. 
 \end{proof}
 \end{subsection}
 \textbf{Acknowledgements:} I would like to thank my advisor Nicolas Templier for his help in improving the manuscript, and to Indranil Biswas and Arideep Saha for helpful comments. In addition, I would like to warmly thank an anonymous referee for many useful suggestions. I acknowledge support by NSF grants DMS-1454893 and DMS-2001071.
\end{section}

  \par\nopagebreak
    
\footnotesize{\bibliography{moduli_stack_G_connections.bib}
\bibliographystyle{alpha}}

    \textsc{Department of Mathematics, Columbia university}, \texttt{af3358@columbia.edu}
\end{document}